\documentclass[leqno, 11pt]{amsart}

\usepackage[top=40mm, bottom=40mm, left=35mm, right=35mm]{geometry}

\usepackage[T1]{fontenc} 
\usepackage[utf8]{inputenc} 
\usepackage[english]{babel}

\usepackage{etoolbox} 

\usepackage{nicefrac} 
\usepackage[dvipsnames]{xcolor} 
\usepackage{mathrsfs} 
\usepackage{amsmath} 
\usepackage{amsfonts} 
\usepackage{amssymb} 
\usepackage[centercolon=false]{mathtools} 

\usepackage{bbm} 
\usepackage{graphicx}
\usepackage{marginnote}
\usepackage[all]{xy}

\usepackage{enumerate} 
\usepackage[shortlabels]{enumitem} 
\usepackage[nameinlink, capitalise, noabbrev]{cleveref}

\usepackage{marginnote}
\usepackage{dsfont}

\usepackage[parfill]{parskip}
\makeatletter
	\def\thm@space@setup{%
		\thm@preskip=\parskip \thm@postskip=0pt
	}
	\makeatother  
\AtBeginEnvironment{proof}{\vspace{-\parskip}}

\DeclareMathAlphabet{\mathpzc}{OT1}{pzc}{m}{it}

\newcommand{\unitary}{{\mathrm{uni}}}
\newcommand{\unitaryPoint}{{\mathrm{uni},\mathrm{pnt}}}
\newcommand{\dual}[1]{{#1}^*}
\newcommand{\unitaryDual}[1]{{#1}^{*,\unitary}}
\newcommand{\spr}{\mathrm{r}} 
\newcommand{\modulus}[1]{\left\lvert #1 \right\rvert}
\newcommand{\norm}[1]{\left\lVert #1 \right\rVert}

\newcommand{\one}{\mathbbm{1}}

\newcommand{\ImpliesProof}[2]{\ref{#1}$\Rightarrow$\ref{#2}:}
\newcommand{\EquivalentProof}[2]{\ref{#1}$\Leftrightarrow$\ref{#2}:}

\newcommand{\notion}[1]{\emph{#1}}

\newcommand{\N}{\mathbb{N}} 
\newcommand{\Z}{\mathbb{Z}}

\newcommand{\R}{\mathbb{R}}
\newcommand{\C}{\mathbb{C}}

\newcommand{\T}{\mathbb{T}}

\newcommand{\fix}{\operatorname{fix}}

\newtheorem{theorem}{Theorem}[section]
\newtheorem{lemma}[theorem]{Lemma}
\newtheorem{corollary}[theorem]{Corollary}
\newtheorem{proposition}[theorem]{Proposition}

\theoremstyle{definition}

\newtheorem{remark}[theorem]{Remark}

\newtheorem{standingassumption}[theorem]{Standing Assumption}
\newtheorem{definition}[theorem]{Definition}
\newtheorem{example}[theorem]{Example}
\newtheorem{examples}[theorem]{Examples}

\setlist[enumerate]{
  topsep = 0pt,
  itemsep = 0.0ex,
  partopsep = 1ex,
  parsep = 1ex,
  labelsep = 0.2cm,
  label = (\alph{enumi}),
  font = \normalfont,
  ref = \labelenumi
}

\setlist[itemize]{
  partopsep = 1ex,
  parsep = 1ex
}

\author{Jochen Glück}
\address{Jochen Glück, School of Mathematics and Natural Sciences, University of Wuppertal, Gaußstraße 20,
D-42119 Wuppertal, Germany}
\email{glueck@uni-wuppertal.de}

\author{Patrick Hermle}
\address{Patrick Hermle, School of Mathematics and Natural Sciences, University of Wuppertal, Gaußstraße 20,
D-42119 Wuppertal, Germany}
\email{patrick.hermle@uni-wuppertal.de}

\author{Henrik Kreidler}
\address{Henrik Kreidler, School of Mathematics and Natural Sciences, University of Wuppertal, Gaußstraße 20,
D-42119 Wuppertal, Germany}
\email{kreidler@uni-wuppertal.de}

\subjclass[2020]{47D03, 47D06, 47B65, 46B42}
\keywords{operator semigroup; long-term behaviour; spectrum of semigroup representations; positive operator; quasi-compactness}

\begin{document}

\title{Uniform ergodic theorems for semigroup representations}

\maketitle

\begin{abstract}
	We consider a bounded representation $T$ of a commutative semigroup $S$ on a Banach space and analyse the relation between three concepts: (i) properties of the unitary spectrum of $T$, which is defined in terms of semigroup characters on $S$; (ii) uniform mean ergodic properties of $T$; and (iii) quasi-compactness of $T$. 
	
	We use our results to generalize the celebrated Niiro--Sawashima theorem to semigroup representations and, as a consequence, obtain the following: if a positive and bounded semigroup representation on a Banach lattice is uniformly mean ergodic and has finite-dimensional fixed space, then it is quasi-compact.
\end{abstract}

\section{Introduction}

\subsection*{Ergodic properties of operators and their powers}

For a power-bounded linear operator $T$ on a Banach space $E$ 
there is an intimate and well-known relation between the following three concepts: 
\begin{enumerate}[(i)]
	\item 
	The \notion{unitary spectrum} of $T$, 
	i.e., of those spectral values $\lambda$ of $T$ in the complex unit circle.
	
	\item 
	\notion{Uniform mean ergodicity} of $T$, i.e., 
	the property that the Cesàro means $\frac{1}{n} \sum_{k=0}^{n-1} T^k$ of the powers of $T$ converge in operator norm as $n \to \infty$.
	
	\item 
	\notion{Quasi-compactness} of $T$, i.e., the existence of a compact operator $K$ and an integer $n \ge 0$ 
	such that $\|T^n - K\| < 1$.
\end{enumerate}
More specifically, $T$ is quasi-compact if and only if 
$T$ is \emph{totally uniformly mean ergodic}, i.e., $\mu T$ is uniformly mean ergodic for every $\mu$ in the complex unit circle, and all eigenspaces of $T$ for unimodular eigenvalues are finite-dimensional. Moreover, this is precisely the case when every spectral value of $T$ with modulus $1$ is a pole of the resolvent with finite-dimensional spectral space.

The case where $E$ is a Banach lattice and $T$ is positive is of particular importance. 
In this case, the power-bounded operator $T$ is quasi-compact if and only if 
it is uniformly mean ergodic with finite-dimensional fixed space, 
meaning that the Cesàro means $\frac{1}{n} \sum_{k=0}^{n-1} T^k$ converge with respect to the operator norm 
to a finite-rank projection as $n \to \infty$. 
This was shown by Lin \cite[Theorem~1]{Lin78} using a theorem of Lotz and Schaefer \cite[Theorem~2]{LoSc1968}
which, in turn, was a generalization of a celebrated spectral theoretic result for positive operators 
due to Niiro and Sawashima \cite[Main Theorem and Theorem~9.2]{NiSa1966}.

\subsection*{Contributions: Ergodic properties of semigroup representations}

In this article we generalize the aforementioned results to bounded representations $T$ of commutative semigroups $S$. 
In contrast to what is quite common in the literature, we do not endow $S$ with any topology and hence, 
we do not impose any continuity assumptions on the mapping $s \mapsto T_s$. 
Or, speaking from a different perspective, we only consider the discrete topology on $S$. 
As a consequence, our results cannot reflect any non-trivial continuity properties of the mapping $s \mapsto T_s$. 
While this might appear as a disadvantage at first glance, 
it fits the philosophy of the recent articles \cite{GerGlu2019, GlHa2019, DoGl2021} 
that, even in the case $S = \big([0,\infty), +\big)$, it is desirable to obtain results about the long-term behaviour of $T$ 
that do not depend on topological properties of $S$ at all.

Our decision to refrain from any topological assumptions on $S$ brings two further advantages:
On the one hand, we do not need to rely on Fourier transforms to define the spectrum of a representation 
and thus, our semigroup $S$ does not need to embed into a group. 
On the other hand, it is very easy within our setting to work with semigroup representations on Banach algebras, 
since there is no need for a strong operator topology. 
This makes it straightforward to apply our theory on the Calkin algebra of a Banach space (\cref{charquasicomp}).

In the first half of the article (Section~\ref{secspec}) we develop a spectral theory for semigroup representations. 
We compare our approach with various related notions from the literature in Remark~\ref{rem:compare-lit}.

In the second half of the article (Section~\ref{secuniform}) we relate the spectrum of a semigroup representation $T$ 
to its long-term behaviour. 
Here we focus in particular on \notion{uniform mean ergodicity} 
(Subsections~\ref{subsec:erg} and~\ref{subsec:total}) and  
on \notion{quasi-compactness} (Subsection~\ref{subsec:quasi-compact}). 
In the final Subsection~\ref{subsec:nisa} we consider positive representations on Banach lattices and 
generalize the celebrated Niiro--Sawashima theorem to such representations.

The entire article is set in the following framework:

\begin{standingassumption}
	\label{assu:standing}
	Throughout the paper we fix a commutative semigroup $(S,+)$, 
	i.e., $S$ is a set and $+ \colon S \times S \rightarrow S$ is an associative operation, 
	and assume that $S$ contains a neutral element $0$. 
	We endow $S$ with a reflexive and transitive relation $\le$ 
	given by $s \le t$ if and only if there exists an element $r \in S$ such that $s+r=t$.
\end{standingassumption}

Note that the relation $\le$ turns $S$ into a directed set: 
for all $s_1, s_2 \in S$ the element $s \coloneqq s_1+s_2$ satisfies $s \ge s_1$ and $s \ge s_2$.

The assumption that $S$ contains a neutral element is only there since it makes some results more convenient to write down. 
For instance, if no neutral element existed, the definition of $\le$ would have to be a bit more involved to ensure reflexivity. 
Note, however, that if a semigroup does not contain a neutral element, one can always adjoin a neutral element to it. 

The semigroup $S$ is called \notion{cancellative} if 
for each $s_0 \in S$ the translation map $S \to S$, $s \mapsto s_0 + s$ is injective. 
It is classical result that the commutative semigroup $S$ embeds into a (commutative) group 
if and only if it is cancellative, see for instance \cite[Theorem~3.10]{Nagy01}.

\subsection*{Convention regarding the scalar field}

All vector spaces throughout the article are complex.

\subsection*{Notation}

We use the convention $\N \coloneqq \{0, 1, 2, \dots\}$. 
Given a Banach space $E$ we write $\mathscr{L}(E)$ for the unital Banach algebra of bounded linear operators on $E$, 
and $E'$ for the dual space of $E$. 
For a unital Banach algebra $A$ we denote its neutral element with respect to multiplication by $\one$. 
For an element $a \in A$ and a number $\lambda \in \C$ 
we sometimes write $\lambda - a$ as a shorthand for $\lambda \one - a$.

\section{Spectral theory for semigroup representations}\label{secspec}

In this first part of the article we consider bounded representations of semigroups and develop a spectral theory for them 
-- more specifically, a theory that focusses on unitary spectral values, 
since those are essential for the long-term behaviour of the semigroup representation.

\subsection{Semigroup representations}

A \notion{representation} $T$ of the semigroup $S$ in a unital Banach algebra $A$ is a monoid homomorphism $T \colon S \rightarrow A$ 
where $A$ is considered as a multiplicative semigroup with neutral element $\one$, 
i.e., $T_{s+t} = T_sT_t$ for all $s,t \in S$ und $T_0 = \one$\color{red}.\color{black} 
A representation $T$ is called \notion{bounded} if $\sup_{s \in S} \|T_s\| < \infty$. 
We note that every represenation $T$ of $S$ in $A$ can be interpreted as a net $(T_s)_{s \in S}$ 
since $S$ is a directed set with respect to the relation $\le$ defined in Standing Assumption~\ref{assu:standing}.

A \notion{semigroup character} of $S$ is a representation of $S$ in the one-dimensional Banach algebra $A = \C$, 
i.e., a monoid homomorphism $\chi \colon S \rightarrow \C$ to the multiplicative semigroup of complex numbers. 
A semigroup character $\chi \colon S \rightarrow \C$ is called a \notion{unitary semigroup character} if it maps into the complex unit circle $\T \coloneqq \{z \in \C \mid |z| = 1\}$.
We write $\dual{S}$ for the set of all semigroup characters of $S$ and $\unitaryDual{S}$ for the set of all unitary semigroup characters and call $\dual{S}$ the \emph{dual} of $S$ and $\unitaryDual{S}$ the \notion{unitary dual} of $S$. 
Equipped with pointwise multiplication and the topology of pointwise convergence, the unitary dual $\unitaryDual{S}$ is a compact group with neutral element  $\mathds{1}_{S} \colon S \rightarrow \T, \, s \mapsto 1$ and inverses $\overline{\chi} \colon S \rightarrow \T, \, s \mapsto\overline{\chi(s)}$ for $\chi \in \unitaryDual{S}$. 
We point out that if $S$ is even a group, the dual group of the (discrete) group $S$ in the sense of harmonic analysis is $\unitaryDual{S}$ (rather than $\dual{S}$).
Observe that for every given representation $T \colon S \rightarrow A$ in a unital Banach algebra $A$ and for every given semigroup character $\chi \in \dual{S}$ the mapping $\chi T \colon S \rightarrow A$, $s \mapsto \chi(s)T_s$ is also a representation of $S$.

We are mostly concerned with more concrete representations of $S$ as bounded operators on a Banach space, i.e., with the case $A = \mathscr{L}(E)$ for some Banach space $E$. 
For such representations $T \colon S \rightarrow \mathscr{L}(E)$ we can consider common standard constructions such as the following:
\begin{itemize}
	\item[(a)] 
	The \notion{restricted representation} $T|_F \colon S \rightarrow \mathscr{L}(E), \, s \mapsto (T_s)|_F$, 
	where $F$ is a closed vector subspace $F\subseteq E$ which is invariant under all operators $T_s$.
	
	\item[(b)] 
	The \notion{dual representation} $T' \colon S \rightarrow \mathscr{L}(E'), \, s \mapsto T_s'$ on the dual Banach space $E'$.
	
	\item[(c)] 
	The \notion{sum} $T_1 \oplus T_2 \colon S \rightarrow \mathscr{L}(E_1 \oplus E_2)$ 
	of two representations $T_1 \colon S \rightarrow \mathscr{L}(E_1)$ and $T_2 \colon S \rightarrow \mathscr{L}(E_2)$ 
	for Banach spaces $E_1, E_2$ on any sum $E_1 \oplus E_2$ is defined in the obvious way.
\end{itemize}

\begin{remark}
	\label{opvsrep}
	Recall that for the additive semigroup $S= \N \coloneqq \{0, 1, 2, \dots\}$, 
	a representation of $S$ in a unital Banach algebra $A$ is completely determined by the element $T_1 \in A$, 
	and conversely every element $a \in A$ induces a representation
	\begin{align*}
		T_a\colon \N \to A,\quad n \mapsto a^n
	\end{align*}
	of the semigroup $\N$ with $(T_a)_1 = a$. 
	So the representations of $\N$ in a Banach algbera $A$ are in one-to-one correspondence with the elements of $A$. 
	
	Moreover, the unitary dual $\unitaryDual{\N}$ of $\N$ and the complex unit circle $\T$ are isomorphic 
	as topological groups via the map $\unitaryDual{\N} \to \T,\, \chi \mapsto \chi(1)$.
	Via these correpondences, the concepts and results of the following sections generalize the existing theory for single operators.
\end{remark}

\subsection{The unitary spectrum}

Based on spectral theory in Banach algebras we introduce some spectral theoretic concepts for bounded semigroup representations.
Recall that for a unital Banach algebra $A$ the \notion{spectrum} of an element $a \in A$ is the compact set
$\sigma(a)  \coloneqq \{\lambda \in \C \mid (\lambda - a) \text{ is not invertible}\}$.
We call $\sigma_{\unitary}(a) \coloneqq \sigma(a) \cap \T$ the \notion{unitary spectrum} of $a$.  
The number 
\begin{align*}
	\spr(a) 
	\coloneqq 
	\max \{\lvert \lambda \rvert : \, \lambda \in \sigma(a)\} 
	= 
	\lim_{n \rightarrow \infty} \|a^n\|^{\frac{1}{n}}
\end{align*}
is called the \notion{spectral radius} of $a \in A$. 
If $a \in A$ is \notion{power-bounded}, i.e., $\sup_{n \in \N} \|a^n\| < \infty$,
then the preceding formula shows that $\spr(a) \le 1$, i.e., the spectrum $\sigma(a)$ is contained in the closed unit disk.

A bounded linear functional $\psi: A \to \C$ is called a \notion{algebra character} if it is non-zero 
and is also a multiplicative semigroup homomorphism, i.e., $\psi(ab) = \psi(a)\psi(b)$ for all $a,b \in A$.
Recall that one automatically has $\psi(\mathbbm{1}) = 1$ for an algebra character $\psi$.
The \notion{Gelfand space} of $A$ is the subset $\Gamma_A$ of the dual Banach space $A'$ that consists of all algebra characters; 
it is a compact set when when equipped with the weak*-topology. 
It is well-known that, if $A$ is commutative, the Gelfand space can be used to characterize the spectra of elements $a \in A$: 
One then has $\sigma(a) = \{\psi(a) \colon \, \psi \in \Gamma_A\}$ for each $a \in A$.

We highlight that the spectrum does depend on the surrounding algebra, 
i.e., if $A$ is a unital Banach algebra, $B$ is a closed unital subalgebra of $A$, and $a \in B$, 
then the spectrum of $a$ can be different in in $A$ and $B$, in general.
If we want to stress the dependence on the surrounding algebra we write $\sigma_A(a)$ rather than $\sigma(A)$ 
and $\sigma_{A,\unitary}(a)$ rather than $\sigma_{\unitary}(a)$ for the unitary spectrum.
However, if $a \in B$ has spectral radius $\spr(a) \le 1$, 
then $\sigma_{A,\unitary}(a) = \sigma_{B,\unitary}(a)$ as a consequence of \cite[Theorem~1.2.8]{Murr1990}. 
We refer to \cite[Chapter~1]{Murr1990} for a general introduction to spectral theory in Banach algebras.

We now want to define the unitary spectrum of a bounded semigroup representation $T: S \rightarrow A$. 
To motivate our approach let us first recall the following characterizations of the spectrum of a power bounded element $a \in A$: 

\begin{proposition}
	\label{prop:spectrum-of-element}
	Let $A$ be a unital Banach algebra and let $a \in A$ be power-bounded.
	For each $\lambda \in \C$ of modulus $\modulus{\lambda} = 1$ the following are equivalent: 
	\begin{enumerate}[start=19, label=\upshape(\alph*)]
		\item\label{prop:spectrum-of-element:itm:spec} 
		One has $\lambda \in \sigma(a)$ 
		(equivalently, $\lambda \in \sigma_{\unitary}(a)$).
		
		\setcounter{enumi}{0}
		
		\item\label{prop:spectrum-of-element:itm:approx} 
		The number $\lambda$ is an \notion{approximate eigenvalue of $a$}, 
		i.e., there exists a sequence $(c_n)$ of normalized elements in $A$ such that
		$(\lambda - a)c_n \to 0$.
		
		\item\label{prop:spectrum-of-element:itm:ideal} 
		The left ideal $A(\lambda-a)$ is not equal to $A$  
		(equivalently: does not contain $\one$).
		
		\item\label{prop:spectrum-of-element:itm:ideal-right} 
		The right ideal $(\lambda-a)A$ is not equal to $A$  
		(equivalently: does not contain $\one$).
		
		\item\label{prop:spectrum-of-element:itm:char-all} 
		For every closed commutative unital subalgebra $B$ of $A$ that satisfies $a \in B$ 
		there exists an algebra character $\psi \in \Gamma_B$ such that $\psi(a) = \lambda$.
		
		\item\label{prop:spectrum-of-element:itm:char-exists} 
		There exists a closed commutative unital subalgebra $B$ of $A$ that satisfies $a \in B$ 
		and an algebra character $\psi \in \Gamma_B$ such that $\psi(a) = \lambda$.
		
		\item\label{prop:spectrum-of-element:itm:laplace} 
		For every complex polynomial $p$ one has $\modulus{p(\lambda)} \le \norm{p(a)}$.
	\end{enumerate}
\end{proposition}

\begin{proof}
	\ImpliesProof{prop:spectrum-of-element:itm:spec}{prop:spectrum-of-element:itm:approx} 
	Since $a$ is power bounded, the spectral radius of $a$ is no more than $1$, 
	so $\lambda$ is in the boundary of $\sigma(a)$. 
	Hence, one can use the following standard argument: 
	let $(\lambda_n)$ be a sequence in $\C \setminus \sigma(a)$ such that $\lambda_n \to \lambda$ 
	and define $c_n \coloneqq (\lambda_n-a)^{-1} / \norm{(\lambda_n-a)^{-1}}$ for each $n$. 
	By using that $\norm{(\lambda_n-a)^{-1}} \to \infty$ one can easily check that $(\lambda-a)c_n \to 0$.
	
	\ImpliesProof{prop:spectrum-of-element:itm:approx}{prop:spectrum-of-element:itm:ideal} 
	If there exists $b \in A$ such that $b(\lambda-a) = \one$ 
	and $(c_n)$ is a normalized sequence, then $b(\lambda-a)c_n = c_n$ for each $n$, 
	so $(\lambda-a)c_n$ cannot converge to $0$ as $n \to \infty$.
	
	\ImpliesProof{prop:spectrum-of-element:itm:ideal}{prop:spectrum-of-element:itm:char-all} 
	It follows from~\ref{prop:spectrum-of-element:itm:ideal} that the ideal $B(\lambda-a)$ in $B$ does not contain $\one$. 
	So there exists a maximal ideal $I$ in $B$ that contains $\lambda-a$. 
	As $B$ is commutative, $I$ is the kernel of a character $\psi \in \Gamma_B$, see \cite[Theorem 1.3.3]{Murr1990}.
	So $\psi(\lambda-a) = 0$ and hence, $\psi(a) = \lambda$.
	
	\ImpliesProof{prop:spectrum-of-element:itm:char-all}{prop:spectrum-of-element:itm:char-exists} 
	This implication is obvious since there exists a closed commutative unital subalgebra of $A$ that contains $a$.
	
	\ImpliesProof{prop:spectrum-of-element:itm:char-exists}{prop:spectrum-of-element:itm:laplace} 
	Let $p$ be a complex polynomial and let $B$ and $\psi$ be as in~\ref{prop:spectrum-of-element:itm:char-exists}. 
	Since $\norm{\psi}=1$ and $p(a) \in B$ we have 
	\begin{align*}
		\modulus{ p(\lambda) } 
		= 
		\modulus{ p(\psi(a)) }
		=
		\modulus{ \psi(p(a)) } 
		\le  
		\norm{p(a)}
		.
	\end{align*}
	
	\ImpliesProof{prop:spectrum-of-element:itm:laplace}{prop:spectrum-of-element:itm:approx} 
	For every integer $n \ge 1$ define the complex polynomial $p_n$ 
	by $p_n(z) = \frac{1}{n}\sum_{k=0}^{n-1} \overline{\lambda}^k z^k$ for all $z \in \C$. 
	One has
	\begin{align*}
		(\lambda - a) p_n(a) 
		= 
		\lambda (\one - \overline{\lambda} a) p_n(a) 
		= 
		\lambda \frac{1}{n} \big( \one - (\overline{\lambda}a)^n \big)
		\to 
		0
	\end{align*}
	as $n \to \infty$ due to the power boundedness of $a$.
	Moreover, it follows from~\ref{prop:spectrum-of-element:itm:laplace} that 
	$\norm{p_n(a)} \ge \modulus{p_n(\lambda)} = 1$. 
	So the normalized elements $c_n \coloneqq p_n(a)/\norm{p_n(a)}$ also satisfy $(\lambda-a)c_n \to 0$, 
	which shows~\ref{prop:spectrum-of-element:itm:approx}.
	
	\ImpliesProof{prop:spectrum-of-element:itm:approx}{prop:spectrum-of-element:itm:spec} 
	If $\lambda \not\in \sigma(a)$ but $(c_n)$ is a normalized sequence such that $(\lambda-a)c_n \to 0$, then 
	\begin{align*}
		c_n 
		= 
		(\lambda-a)^{-1} (\lambda-a) c_n 
		\to 
		0
		,
	\end{align*} 
	which is a contradiction. 
	
	\EquivalentProof{prop:spectrum-of-element:itm:spec}{prop:spectrum-of-element:itm:ideal-right}
	We can turn $A$ into a new unital Banach algebra $\tilde A$ by endowing it with the 
	new multplication $\star$ given by $a \star b \coloneqq ba$ for all $a,b \in A$. 
	Clearly, every element $a \in A$ has the same spectrum in $A$ and in $\tilde A$, 
	and power boundedness of an element is also the same in both algebras.
	So we obtain the claimed equivalence by applying the equivalence of~\ref{prop:spectrum-of-element:itm:spec} 
	and~\ref{prop:spectrum-of-element:itm:ideal} in $\tilde A$.
\end{proof}

Let $A$ be a unital Banach algebra. 
Recall from Remark~\ref{opvsrep} that every $a \in A$ corresponds to a semigroup representation of $\N = \{0,1,2,\dots\}$ given by 
$n \mapsto a^n$ and that the unitary dual $\unitaryDual{\N}$ is isomorphic (as a topological group) to the unit cirlce $\T$. 
Hence, it is natural to define the \notion{unitary spectrum} of a bounded semigroup representation $T \colon S \to A$ 
as a subset of the unitary dual $\unitaryDual{S}$.  
In fact, we will now show that the assertions~\ref{prop:spectrum-of-element:itm:approx}--\ref{prop:spectrum-of-element:itm:laplace} 
from Proposition~\ref{prop:spectrum-of-element} are still equivalent when they are adapted to the setting of 
general bounded semigroup representations $T$. 
This motivates how we define the spectrum of $T$ in Definition~\ref{def:spectrum}. 

\begin{proposition}
	\label{prop:charspec}
	Let $T \colon S \rightarrow A$ be a bounded representation in a unital Banach algebra $A$. 
	For every unitary character $\chi \in \unitaryDual{S}$ the following are equivalent.
	\begin{enumerate}[(a)]
		\item\label{prop:charspec:itm:approx} 
		There is a net $(c_j)$ in $A$ with $\|c_j\| = 1$ for every $j$ such that
		\begin{align*}
			 (\chi(s) - T_s)c_j \overset{j}{\to} 0
		\end{align*}
		for every $s \in S$.
		
		\item\label{prop:charspec:itm:ideal} 
		The left ideal in $A$ generated by the set $\{\chi(s) - T_s \colon s \in S\}$ is not equal to $A$, 
		i.e., there are no $a_1, \dots, a_n \in A$ and $s_1, \dots, s_n \in S$ such that
		\begin{align*}
			\sum_{k=1}^n a_k (\chi(s_k) - T_{s_k}) = \one.
		\end{align*}
		
		\item\label{prop:charspec:itm:ideal-right} 
		The right ideal in $A$ generated by the set $\{\chi(s) - T_s \colon s \in S\}$ is not equal to $A$, 
		i.e., there are no $a_1, \dots, a_n \in A$ and $s_1, \dots, s_n \in S$ such that
		\begin{align*}
			\sum_{k=1}^n (\chi(s_k) - T_{s_k}) a_k = \one.
		\end{align*}
		
		\item\label{prop:charspec:itm:char-all}
		For every closed commutative unital subalgebra $B$ of $A$ that satisfies $T(S) \subseteq B$ 
		there exists an algebra character $\psi \in \Gamma_B$ such that $\chi = \psi \circ T$.
		
		\item\label{prop:charspec:itm:char-exists} 
		There exists a closed commutative unital subalgebra $B$ of $A$ that satisies $T(S) \subseteq B$ 
		and an algebra character $\psi \in \Gamma_B$ such that $\chi = \psi \circ T$.
		
		\item\label{prop:charspec:itm:laplace} 
		For all $s_1, \dots, s_n \in S$ and $\beta_1, \dots, \beta_n \in \C$ one has 
		\begin{align*}
			\modulus{\sum_{k=1}^n \beta_k \chi(s_k)} 
			\le 
			\norm{\sum_{k=1}^n \beta_k T_{s_k}}
			.
		\end{align*}
	\end{enumerate}
\end{proposition}

Our proof that~\ref{prop:charspec:itm:approx} implies~\ref{prop:charspec:itm:ideal} 
is inspired by the one of~\cite[Proposition 2.2]{BaPh1992}. 
We need the following terminology and the subsequent lemma. 
An \notion{ergodic net} for a representation $T \colon S \rightarrow A$ in a unital Banach algebra 
is a net $(c_j) \subseteq \overline{\mathrm{co}}\, T(S) \subseteq A$ such that 
\begin{align*}
	\lim_j \, (\mathbbm{1}-T_s)c_j = 0  \text{ for every } s \in S,
\end{align*}
see~\cite[p.\,87]{Kren1985}. 
The following can be shown as in \cite[pp.\,75--76]{Kren1985}. 
	
\begin{lemma}
	\label{ergodicnet}
	Every bounded representation $T \colon S \rightarrow A$ in a unital Banach algebra has an ergodic net.
\end{lemma}

Let us now show Proposition~\ref{prop:charspec}. 

\begin{proof}[Proof of \cref{prop:charspec}]
	\ImpliesProof{prop:charspec:itm:approx}{prop:charspec:itm:ideal} 
	Let $(c_j)$ be a net of normalized elements in $A$ that satisfies $(\chi(s) - T_s)c_j \to 0$ for each $s \in S$.
	If there are $a_1, \dots, a_n \in A$ and $s_1, \dots, s_n \in S$ such that
	\begin{align*}
		\sum_{k=1}^n a_k (\chi(s_k) - T_{s_k}) = \one,
	\end{align*}
	then multiplying this equality from the right with $c_j$ yields $c_j \to 0$, a contradiction.

	\ImpliesProof{prop:charspec:itm:ideal}{prop:charspec:itm:char-all} 
	Consider a closed commutative unital subalgebra $B$ that satisfies $T(S) \subseteq B$. 
	It follows from~\ref{prop:charspec:itm:ideal} 
	that the ideal in $B$ generated by the set $\{\chi(s) - T_s \colon s \in S\}$ is not equal to $B$, 
	so it is contained in a maximal ideal $I$ of $B$.
	Hence, there exists an algebra character $\psi_B \in \Gamma_B$ that vanishes on $I$, see e.g.~\cite[Theorem 1.3.3]{Murr1990}.
	Therefore, one has $\psi(\chi(s) - T_s) = 0$ and thus $\psi(T_s) = \chi(s)$ for each $s \in S$.

	\ImpliesProof{prop:charspec:itm:char-all}{prop:charspec:itm:char-exists} 
	This implication is clear since $T(S)$ is commutative and thus contained in a closed unital subalgebra of $A$.

	\ImpliesProof{prop:charspec:itm:char-exists}{prop:charspec:itm:laplace} 
	Assume that \ref{prop:charspec:itm:char-exists} holds and choose $B$ and $\psi \in \Gamma_B$ as in \ref{prop:charspec:itm:char-exists}. 
	For all $s_1, \dots, s_n \in S$ and $\beta_1, \dots, \beta_n \in \C$ one then gets
	\begin{align*}
		\modulus{\sum_{k=1}^n \beta_k \chi(s_k)} 
		= 
		\modulus{\sum_{k=1}^n \beta_k \psi(T_{s_k})} 
		\le 
		\norm{\sum_{k=1}^n \beta_k T_{s_k}}
	\end{align*}
	since $\norm{\psi} \le 1$.

	\ImpliesProof{prop:charspec:itm:laplace}{prop:charspec:itm:approx}
	For all convex coefficients $\mu_1, \dots, \mu_n$ and all $s_1, \dots, s_n$, 
	applying~\ref{prop:charspec:itm:laplace} 
	to the coefficients $\beta_k \coloneqq \mu_k \overline{\chi}(s_k)$ for $k \in \{1, \dots, n\}$ gives
	\begin{align*}
		1 \le \norm{\sum_{k=1}^n \mu_k \overline{\chi}(s_k) T_{s_k}}.
	\end{align*} 
	So if $(d_j) \subseteq A$ is an ergodic net for the rotated representation $\overline{\chi} T \colon S \rightarrow A$ 
	-- which exists by \cref{ergodicnet} -- 
	then $1 \le \norm{d_j}$ for each index $j$.
	Moreover,
	\begin{align*}
		\|(\chi(s) - T_s)d_j)\| = \|(\mathbbm{1} - (\overline{\chi}T)_s)d_j)\| \rightarrow 0
	\end{align*}
	for every $s \in S$. 
	By setting $c_j \coloneqq d_j/\norm{d_j}$ for each $j$ 
	we thus obtain the desired net in~\ref{prop:charspec:itm:approx}.

	\EquivalentProof{prop:charspec:itm:char-all}{prop:charspec:itm:ideal-right} 
	We turn $A$ into a new unital Banach algebra $\tilde A$ by endowing it with the 
	multplication $\star$ defined by $a \star b \coloneqq ba$ for all $a,b \in A$. 
	Since $S$ is commutative, $T$ is also a representation of $S$ in $\tilde A$. 
	Property~\ref{prop:charspec:itm:char-all} of $T$ within the algebra $A$ 
	is equivalent to the same property within the algebra $\tilde A$. 
	Hence, we get the claimed equivalence by applying the equivalence 
	of~\ref{prop:charspec:itm:char-all} and~\ref{prop:charspec:itm:ideal} in the algebra $\tilde A$.
\end{proof}

Propositions~\ref{prop:spectrum-of-element} and~\ref{prop:charspec} motivate the following definition 
of the unitary spectrum of a bounded semigroup representation.

\begin{definition}
	\label{def:spectrum}
	Let $A$ be a unital Banach algebra and $T \colon S \rightarrow A$ a bounded representation. 
	The \notion{unitary spectrum} of $T$ is the set $\sigma_{\unitary}(T)$ 
	of all unitary semigroup characters $\chi \in \unitaryDual{S}$ 
	that satisfy the equivalent conditions of \cref{prop:charspec}.
\end{definition}

\begin{remark}\label{rem:compare-lit}
	Generalizations of spectral theory to semigroup and group representations are a classical theme in operator theory:
	\begin{enumerate}[(i)]
		\item 
		For a locally compact Abelian group $G$ and a Banach space $E$, 
		Arveson defined the spectrum of a bounded representation $T \colon G \to \mathscr{L}(E)$
		by lifting $T$ to a representation of the convolutional Banach algebra $L^1(G)$ 
		and using the Fourier transform \cite[Definition~2.1]{Arveson74}.
		
		A definition that also applies to non-bounded group representations was given in terms of approximate eigenvectors 
		by Lyubich \cite{Lyubich71} (see also \cite{LyMaFe73}). 
		This is close in spirit to property~\ref{prop:charspec:itm:approx} in \cref{prop:charspec} 
		and also to Proposition~\ref{charspec2} below.
		
		Those references focus on the case of group representations, while our main interest is in semigroup representations.
		
		\item 
		Wolff used non-standard analysis to study \notion{Riesz points} of the spectrum \cite[Sections~5 and~6]{Wolff1984}. 
		This is related to our Subsections~\ref{upsubsec}, \ref{subsec:quasi-compact} and~\ref{subsec:nisa}, 
		although the focus in \cite{Wolff1984} is also on group rather than semigroup representations.
		
		\item
		Spectral theory for positive group representations on Banach lattices 
		-- in other words, \notion{Perron--Frobenius theory} for group representations -- 
		was studied by Greiner and Groh in \cite{GrGr83} and by Wolff in \cite[Section~6]{Wolff1984}.
		A similar theme occurs for semigroups in Subsection~\ref{subsec:nisa}.
		
		\item 
		For a representation of a commutative Banach algebra $B$ in a Banach space $E$, 
		three different spectra $\Lambda_1$, $\Lambda_2$, and $\Lambda_3$ 
		are compared by Domar and Lindahl in \cite{DoLi75}. 
		If one lifts a representation $T \colon G \to \mathscr{L}(E)$ of a locally compact Abelian group $G$ 
		to a representation $L^1(G) \to \mathscr{L}(E)$ of the convolutional algebra $B \coloneqq L^1(G)$, 
		the spectrum $\Lambda_1$ from \cite{DoLi75} 
		agrees with the aforementioned one defined by Arveson \cite[Definition~2.1]{Arveson74}. 
		The second spectrum $\Lambda_2$ from \cite{DoLi75} is obtained 
		as a subset of the first one by imposing an additional boundedness condition. 
		The third spectrum $\Lambda_3$ is defined in terms of algebra characters on $B$ and approximate eigenvectors. 
		
		\item 
		The spectrum of bounded semigroup -- rather than group -- representations was studied 
		by Batty and Phóng \cite{BaPh1992} and by Huang \cite{Huang95}.
		Batty and Phóng were mainly interested in the long-term behaviour of representations with respect to the strong operator topology. 
		Huang has several notions and results about asymptotics with respect to the operator norm 
		that are closer to the contents of the present article; 
		see in particular our comments after \cref{def:pole-riesz} and \cref{chartotallyuniform}.
		
		The assumptions on the semigroup $S$ in \cite{BaPh1992, Huang95} are somewhat different from ours. 
		On the one hand, the authors of \cite{BaPh1992, Huang95} allow $S$ to carry a topology 
		and are thus able to take consequences of strong continuity of representations into account. 
		On the other hand, they assume $S$ to be embedded into a locally compact Abelian group $G$ and to satisfy $S-S = G$, 
		along with a topological and a measure theoretic assumption. 
		Under similar conditions, spectral theory was later studied also for unbounded semigroup representations 
		by Basit and Pryde in \cite{BaPr04}.
		
		In our case -- where $S$ carries the discrete topology -- it is easy to check that the assumptions from \cite{BaPh1992, Huang95}
		are satisfied if and only if $S$ embeds algebraically into a group, i.e., if and only if $S$ is cancellative 
		(see the discussion after Standing Assumption~\ref{assu:standing}). 
		In this respect, our setting is more general since we do not need $S$ to be cancellative 
		and hence, representations as the one in Example~\ref{exa:intersection-sg} below are covered by our theory. 
		
		The spectrum is defined in \cite{BaPh1992, Huang95} in terms of the Fourier transform on $G$. 
		If our semigroup $S$ is cancellative
		one can apply the definition of the unitary spectrum from \cite{BaPh1992} and \cite{Huang95} 
		(note that in the latter reference, the unitary spectrum is called \notion{peripheral spectrum}).
		It follows from \cref{prop:charspec}\ref{prop:charspec:itm:laplace} that this definition then coincides with ours.
	\end{enumerate}
\end{remark}

The simplest example in which our spectral theoretic notion can be applied is, of course, 
a representation of the additive semigroup $\N$:

\begin{example}
	Let $a \in A$ be a power-bounded element of a unital Banach algebra $A$ 
	and $T_a \colon \N \rightarrow A$ the induced semigroup representation from \cref{opvsrep}. 
	By \cref{prop:spectrum-of-element} 
	the canonical isomorphism $\unitaryDual{\N} \rightarrow \T$ 
	restricts to a homeomorphism $\sigma_{\unitary}(T_a) \rightarrow \sigma_{\unitary}(a)$.
\end{example}

Further and more interesting examples will be discussed throughout the article.

Note that the unitary spectrum of a bounded representation $T \colon S \rightarrow A$ 
can be computed in any closed unital subalgebra $B$ of $A$ containing $T(S)$ and does not depend on the choice of $B$.
We write $\mathrm{A}(T(S))$ for the smallest closed unital subalgebra containing $T(S)$ 
and note that $\mathrm{A}(T(S))$ is always commutative. 
The elements of the unitary spectrum $\sigma_{\unitary}(T)$ are in one-to-one correspondence 
with those algebra characters on $\mathrm{A}(T(S))$ that only take on unitary values on $T(S)$:

\begin{proposition}
	\label{homeomorph}
	For a bounded representation $T \colon S \rightarrow A$ in a unital Banach algebra $A$ the map
	\begin{align*}
		\big\{\psi \in \Gamma_{\mathrm{A}(T(S))}\mid |\psi(T_s)| = 1 \text{ for all } s \in S \big\} 
		& \rightarrow 
		\sigma_{\unitary}(T), \\ 
		\psi 
		& \mapsto 
		\psi \circ T
	\end{align*}
	is a homeomorphism. 
	In particular, $\sigma_{\unitary}(T)$ is a closed subset of $\unitaryDual{S}$.
\end{proposition}

\begin{proof}
	It is clear that the map is injective and continuous. 
	Moreover, it is surjective by definition of $\sigma_\unitary(T)$. 
	Since the Gelfand space $\Gamma_{\mathrm{A}(T(S))}$ is compact and $\sigma_\unitary(T) \subseteq \unitaryDual{S}$ is a Hausdorff space, 
	the map is a homeomorphism. 
	Therefore, $\sigma_{\unitary}(T)$ is also compact and thus closed in the Hausdorff space $\unitaryDual{S}$.
\end{proof}

A very similar result to \cref{homeomorph}, though under somewhat different assumptions on $S$, 
can be found in \cite[Proposition~2.4]{BaPh1992}.

Let $T \colon S \to A$ is a bounded representation in a unital Banach algebra $A$. 
If a semigroup character $\chi \in \unitaryDual{S}$ is in the unitary spectrum $\sigma_\unitary(T)$, 
then it follows from \cref{prop:charspec}\ref{prop:charspec:itm:ideal}  
that $\chi(s) \in \sigma(T(s))$ for each $s \in S$ 
(where $\sigma(T(s)) \subseteq \C$ denotes the usual spectrum of the element $T(s) \in A$).
It is natural to ask whether the converse is true: 
If $\chi \in \unitaryDual{S}$ satisfies $\chi(s) \in \sigma(T(s))$ for each $s \in S$, 
does it follow that $\chi \in \sigma_\unitary(T)$? 

The following example show that the answer is negative. 
Another counterexample -- where $S$ is even finite and $A$ is finite dimensional -- 
is given in Example~\ref{exa:klein-four} below.

\begin{example}
	\label{exa:functions-on-circle}
	Let $A$ be the unital commutative Banach algebra $\mathrm{C}(\T)$ (which is even a $\mathrm{C}^*$-algebra) 
	and let $S$ denote the additive semigroup $\N \times \N$. 
	Let the elements $f,g \in A$ be given by 
	\begin{align*}
		f(z) = z 
		\qquad \text{and} \qquad 
		g(z) = -z
	\end{align*}
	for all $z \in \T$ and define the bounded representation 
	\begin{align*}
		T \colon S     & \to A          , \\
		         (m,n) & \mapsto f^m g^n.
	\end{align*}
	Consider the constant semigroup character $\mathds{1}_S \in \unitaryDual{S}$. 
	For each $s = (m,n) \in S$ one has $\mathds{1}_S(s) = 1 \in \sigma(f^m g^n) = \sigma(T(s))$ 
	since $(f^m g^n)(z) = (-1)^n z^{m+n}$ for all $z \in \T$ and the range of $f^m g^n$ thus contains the point $1$.
	Yet, we now show that $\mathds{1}_S \not\in \sigma_\unitary(T)$. 
	
	Indeed, assume to the contrary that $\mathds{1}_S \in \sigma_\unitary(T)$. 
	By Proposition~\ref{prop:charspec}\ref{prop:charspec:itm:char-all} 
	there exists an algebra character $\psi \in \Gamma_{\mathrm{C}(\T)}$ such that $\mathds{1}_S = \psi \circ T$. 
	Since every algebra character on $\mathrm{C}(\T)$ is a point evaluation, 
	we can thus find a point $z_0 \in \T$ such that $1 = (f^m g^n)(z_0) = (-1)^n z_0^{m+n}$ for all $m,n \in \N$. 
	For $m=1$ and $n=0$ this gives $z_0 = 1$, but for $m=0$ und $n=1$ it gives $z_0 = -1$, 
	so we arrived at a contradiction.
\end{example}

In the next theorem we use \cref{homeomorph}
to characterize the stability of semigroup representations in terms of their unitary spectrum. 
Recall that the semigroup $S$ is a directed set endowed with the relation $\le$ from Standing Assumption~\ref{assu:standing}. 
Hence, every semigroup representation $T$ of $S$ is a net $(T_s)_{s \in S}$ 
and whenever we talk about convergence of $T$ we mean convergence of this net.

\begin{theorem}
	\label{stability}
	For a bounded representation $T \colon S \rightarrow A$ in a unital Banach algebra $A$ the following assertions are equivalent.	
	\begin{enumerate}[(a)]
		\item \label{stabilitya} 
		$\sigma_{\unitary}(T) = \emptyset$.
		
		\item \label{stabilityb} 
		There is an $s \in S$ with $\lim T_s^n = 0$ in $A$.
		
		\item \label{stabilityc} 
		$0 \in \overline{T(S)} \subseteq A$.
		
		\item \label{stability-conv}
		The net $(T_s)_{s \in S}$ converges to $0$ with respect to the operator norm.
		
		\item \label{stabilityd} 
		There is an $s \in S$ with $\|T_s\| < 1$.
	\end{enumerate}
\end{theorem}

\begin{proof}
	\ImpliesProof{stabilitya}{stabilityb}
	First note that, by the boundedness of $T$, 
	one has $\modulus{\psi(T_{s})} \le 1$ for every $\psi \in \Gamma_{\mathrm{A}(T(S))}$ and every $s \in S$. 
	Now assume that~\ref{stabilitya} holds. 
	Then it follows from \cref{homeomorph} and from the compactness of the Gelfand space $\Gamma_{\mathrm{A}(T(S))}$ 
	that we can find $s_1, \dots, s_m \in S$ such that $\min_{i=1, \dots, m} \modulus{\psi(T_{s_i})} < 1$ 
	for every $\psi \in \Gamma_{\mathrm{A}(T(S))}$. 
	For $s \coloneqq s_1 + \dots + s_m \in S$ we then have 
		\begin{align*}
			\lim_{n \rightarrow \infty} \|T_s^n\|^{\frac{1}{n}} 
			= 
			\spr(T_s)
			=
			\max \big\{ \modulus{\psi(T_s)} \mid \psi \in  \Gamma_{\mathrm{A}(T(S))} \big\} < 1,
		\end{align*}
	which implies $\lim_{n \rightarrow \infty} T_s^n = 0$ in $A$. 

	\ImpliesProof{stabilityb}{stabilityc}
	This implication is obvious. 
	
	\ImpliesProof{stabilityc}{stability-conv}
	As $T$ is bounded there exists a number $M \ge 0$ such that $\lVert T_s \rVert \le M$ for all $s \in S$. 
	Now let $\varepsilon > 0$. 
	By~\ref{stabilityc} there exists an $s_0 \in S$ such that $\lVert T_{s_0} \rVert \le \varepsilon$. 
	Hence one has $\lVert T_{s_0+r} \rVert \le M \varepsilon$ for all $r \in S$, 
	so $\lVert T_s \rVert \le M \varepsilon$ for all $s \ge s_0$.
	
	\ImpliesProof{stability-conv}{stabilityd}
	This implication is obvious.
	
	\ImpliesProof{stabilityd}{stabilitya}
	Assume that~\ref{stabilityd} holds and take $s \in S$ with $\|T_s\| < 1$. 
	Then $|\psi(T_s)| \leq \|T_s\| <1$ for every $\psi \in \Gamma_{\mathrm{A}(T(S))}$ 
	and hence $\sigma_\unitary(T) = \emptyset$ by \cref{homeomorph}.
\end{proof}

We now focus on bounded representations of $S$ on Banach spaces. 
In this case, the elements of the unitary spectrum can also be characterized as follows.

\begin{proposition}
	\label{charspec2}
	Let $E$ be a Banach space. 
	For a bounded representation $T \colon S \rightarrow \mathscr{L}(E)$ 
	and a unitary semigroup character $\chi \in \unitaryDual{S}$ the following are equivalent.
	\begin{enumerate}[(a)]
		\item \label{charspec2a} 
		$\chi \in \sigma_\unitary(T)$.
		
		\item \label{charspec2b} 
		There is a net $(x_j)_{j \in J}$ in $E$ with $\|x_j\| = 1$ for all $j \in J$ such that
		\begin{align*}
			\lim_{j \in J} \|\chi(s)x_j - T_sx_j\| = 0 \text{ for every } s \in S.
		\end{align*}

		\item\label{charspec2c} 
		For every finite subset $\emptyset \not= M \subseteq S$ and every $\varepsilon > 0$ there is a normalized vector $x \in E$ with
		\begin{align*}
			\max_{s \in M} \|\chi(s)x - T_{s}x\| \leq \varepsilon.
		\end{align*}
	\end{enumerate}
\end{proposition}

A net $(x_j)_{j \in J}$ as in \cref{charspec2}\ref{charspec2b} 
is called an \notion{approximate eigenvector} for $\chi \in \sigma_\unitary(S)$.
If $S$ is cancellative the equivalence of properties~\ref{charspec2a} and~\ref{charspec2b} 
in the proposition also follows from \cite[Proposition~2.2]{BaPh1992}.

\begin{proof}[Proof of \cref{charspec2}]
	\ImpliesProof{charspec2a}{charspec2b}
	This implication is a straightforward consequence of property \ref{prop:charspec:itm:approx} in \cref{prop:charspec}.

	\ImpliesProof{charspec2b}{charspec2a}
	Let $(x_j)_{j \in J} \subseteq E$ be a an approximate eigenvector as in~\ref{charspec2b}.
	Choose a functional $x' \in E'$ of norm $1$.  
	Then the net $(x'\otimes x_j)_{j\in J} \subseteq \mathscr{L}(E)$ of rank-one operators given by 
	\begin{align*}
		x'\otimes x_j \colon y \mapsto \langle x', y \rangle x_j
	\end{align*}
	satisfies the property~\ref{prop:charspec:itm:approx} in \cref{prop:charspec}.

	\ImpliesProof{charspec2b}{charspec2c} 
	This implication is obvious.

	\ImpliesProof{charspec2c}{charspec2b}
	Assume that~\ref{charspec2c} holds. 
	We consider the set $\mathscr{P}_{\mathrm{fin}}(S)$ of non-empty finite subsets of $S$ ordered by set inclusion. 
	Endow $\N$ with its natural order and $J \coloneqq \mathscr{P}_{\mathrm{fin}}(S) \times \N$ with the product order. 
	By choosing for every $(M,n) \in J$ some $x_{M,n} \in E$ with $\|x_{M,n}\| = 1$ 
	and $\max_{s \in M} \|\chi(s)x_{n,M} - T_sx_{n,M}\| \leq \frac{1}{n}$ we obtain a net 
	with the properties required in~\ref{charspec2b}. 
\end{proof}

The following are some simple examples for the unitary spectrum of semigroup representations on operator spaces $\mathscr{L}(E)$.

\begin{examples}\label{examplesspec}
	\begin{enumerate}[(i)]
		\item \label{examplesspec1} 
		Let $E$ be a Banach space 
		and let $T \colon \R_{\geq 0} \rightarrow \mathscr{L}(E)$ be 
		a strongly continuous bounded semigroup representation -- i.e., a bounded $C_0$-semigroup -- 
		with generator $A$ which is eventually norm-continuous (see \cite[Subsection~II.4.c]{EN00}), 
		e.g., an analytic $C_0$-semigroup. 
		For every $r \in \mathrm{i}\R$, define the unitary semigroup character $\chi_r \in \unitaryDual{S}$ 
		by $\chi_r(t) \coloneqq \mathrm{exp}(rt)$ for all $t \geq 0$.
		We claim that 
		\begin{align*}
			\psi \colon 
			\sigma(A) \cap \mathrm{i}\R & \rightarrow \sigma_\unitary(T), 
			\\
			\qquad 
			r & \mapsto \chi_r
		\end{align*}
		is a homeomorphism.
		
		First, we show that $\psi$ is well-defined, i.e.\ that it indeed maps into $\sigma_\unitary(T)$. 
		So let $r \in \sigma(A) \cap \mathrm{i}\R$. 
		Then we find an approximate eigenvector $(x_j)_{j \in J}$ of $A$ for $r$ since, 
		by the boundedness of $T$, $r$ is in the boundary of $\sigma(A)$. 
		But then $(x_j)_{j \in J}$ is also an approximate eigenvector for $\chi_r$ (see the proof of \cite[Theorem IV.3.6]{EN00}), 
		hence $\chi_r \in \sigma_\unitary(T)$. 
		So $\psi$ is indeed well-defined.
		
		It is obvious that $\psi$ is continuous and injective. 
		Since $\sigma(A) \cap \mathrm{i}\R$ is compact by \cite[Theorem~II.4.18]{EN00}, 
		continuity of the inverse map is automatic once we have established that the map is surjective.

		To do so, take $\chi \in \sigma_\unitary(T)$.
		It follows from the eventual norm continuity of $T$ and from \cref{prop:charspec}\ref{prop:charspec:itm:char-exists}, 
		that $\chi$ is also eventually continuous.
		As $\chi$ is unitary, we conclude that $\psi$ is even continuous. 
		Thus, there exists $r \in \mathrm{i}\R$ such that $\psi = \psi_r$. 
		We need to show that $r \in \sigma(A) \cap \mathrm{i}\R$.
		
		To this end, let $c \coloneqq \max\{|z| \mid z \in \sigma(A) \cap \mathrm{i}\R\} +1$. 
		Since 
		\begin{align*}
			\exp\biggl(\frac{2\pi r}{c+|r|}\biggr) 
			= 
			\chi_r \biggl(\frac{2\pi }{c+|r|}\biggr) 
			\in 
			\sigma\biggl(T\biggl(\frac{2\pi}{c+|r|}\biggr)\biggr),
		\end{align*}
		we find by the spectral mapping theorem for eventually norm continuous semigroups 
		(see \cite[Theorem~IV.3.10]{EN00}) 
		some $s \in \sigma(A) \cap \mathrm{i}\R$ such that 
		\begin{align*}
			\exp\biggl(\frac{2\pi r}{c+|r|}\biggr) 
			= 
			\exp\biggl(\frac{2\pi s}{c+|r|}\biggr)
			.
		\end{align*}
		Hence $s-r = k\mathrm{i} (c+|r|)$ for some $k \in \Z$. If $k \neq 0$, then $|s| \geq |s-r| - |r|\geq c$, a contradiction. 
		Thus, $r = s \in \sigma(A) \cap \mathrm{i}\R$.

		\item \label{examplesspec2} 
		Assume that $S = G$ is an abelian group and 
		consider the regular representation $T \colon G\rightarrow \mathscr{L}(\ell^1(G))$ 
		on the space $\ell^1(G)$ of absolutely summable complex-valued functions on $G$ 
		defined by $T_s (a_t)_{t \in G} \coloneqq (a_{s+t})_{t \in G}$ 
		for all $(a_t)_{t \in G} \in \ell^1(G)$ and $s \in G$. 
		Then $\sigma_\unitary(T) = \unitaryDual{G}$. 
		
		To see this, note that $\ell^1(G)$ can be interpreted 
		as a closed commutative unital subalgebra of $\mathscr{L}(\ell^1(G))$ by means of convolution, 
		and that this subalgebra contains $T(S) = T(G)$. 
		Hence, the claim follows from \cref{prop:charspec}\ref{prop:charspec:itm:char-all}.
	\end{enumerate}
\end{examples}

We need the following compatibility properties of the unitary spectrum with standard constructions.

\begin{lemma}\label{compatibility}
	For a Banach spaced $E$ 
	and a bounded representation $T \colon S \rightarrow \mathscr{L}(E)$ 
	the following assertions hold.
	\begin{enumerate}[(i)]
		\item \label{compatibility1}
		$\sigma_\unitary(T|_F) \subseteq \sigma_\unitary(T)$ for every closed invariant vector subspace $F \subseteq E$.
		
		\item \label{compatibility2}
		$\sigma_\unitary(T') = \sigma_\unitary(T)$ for the dual representation $T'$ on $E'$.
		
		\item \label{compatibility3}
		$\sigma_\unitary(\chi T) = \chi \cdot \sigma_\unitary(T)$ for every unitary semigroup character $\chi \in \unitaryDual{S}$.
		
		\item \label{compatibility4}
		$\sigma_\unitary(T_1 \oplus T_2) = \sigma_\unitary(T_1) \cup \sigma_\unitary(T_2)$ 
		for every sum $T_1 \oplus T_2$ of bounded representations 
		$T_1 \colon S \rightarrow \mathscr{L}(E_1)$ and $T_2 \colon S \rightarrow \mathscr{L}(E_2)$ on two Banach spaces $E_1, E_2$.
	\end{enumerate}
\end{lemma}

\begin{proof}
	\ref{compatibility1}, \ref{compatibility3}: 
	Those assertions can readily checked by using \cref{charspec2}\ref{charspec2b}.
	
	\ref{compatibility2}: 
	This is an easy consequence of the fact that the restriction of $\mathscr{L}(E) \rightarrow \mathscr{L}(E')$, $R \mapsto R'$ 
	to any commutative unital subalgebra of $\mathscr{L}(E)$ is an isometric unital algebra homomorphism. 
	
	\ref{compatibility4} 
	The inclusion $\supseteq$ is an immediate consequence of part~\ref{compatibility1}. 
	The converse inclusion $\subseteq$ also follows from \cref{charspec2}\ref{charspec2b} 
	by using the following observation:
	if a directed set $J$ is the disjoint union of two sets $J_1$ and $J_2$, 
	then at least one of the sets $J_1, J_2$ is majorizing in $J$.
\end{proof}

Similarly as for single operators, we can also consider the point spectrum of a bounded semigroup representation.

\begin{definition}
	Let $E$ be a Banach space.
	For a bounded representation $T \colon S \rightarrow \mathscr{L}(E)$ 
	we call a unitary semigroup character $\chi \in \unitaryDual{S}$ an \notion{eigenvalue} of $T$ 
	if there exists an \notion{eigenvector} with respect to $\chi$, i.e., 
	a non-zero vector $x \in E$ such that $T_sx=\chi(s)x$ for every $s \in S$. 
	The set $\sigma_\unitaryPoint(T)$ of all such eigenvalues $\chi \in \unitaryDual{S}$ 
	is called the \notion{unitary point spectrum} of $T$. 
	Moreover, we call
	\begin{align*}
		\ker(\chi - T) 
		\coloneqq 
		\bigcap_{s \in S} \ker(\chi(s) - T_s) 
		= 
		\{x \in E \mid T_sx = \chi(s)x \textrm{ for every } s \in S\}
	\end{align*}
	the \notion{eigenspace} of $T$ with respect to $\chi \in \unitaryDual{S}$. 
	The space
	\begin{align*}
		\fix(T) 
		\coloneqq 
		\ker(\mathds{1}_S - T) 
		= 
		\{x \in E \mid T_sx = x \textrm{ for every } s \in S\}
	\end{align*}
	is called the \notion{fixed space} of $T$.
\end{definition}

It follows from \cref{charspec2}\ref{charspec2b} that the unitary point spectrum of a bounded representation $T$ 
is contained in the unitary spectrum of $T$, i.e.\ $\sigma_\unitaryPoint(T) \subseteq \sigma_\unitary(T)$.

\begin{example}
	\label{exa:C_0-pnt-spec}
	Let $E$ be a Banach space.  
	Consider a bounded strongly continuous representation $T \colon \R_{\geq 0} \rightarrow \mathscr{L}(E)$ 
	which is eventually norm continuous and let $A$ denote its generator. 
	Then the map from Example \ref{examplesspec} \ref{examplesspec1} restricts to a homoemorphism
	\begin{align*}
		\sigma_{\mathrm{pnt}}(A) \cap \mathrm{i}\R \rightarrow \sigma_\unitaryPoint(T), \quad r \mapsto \chi_r,
	\end{align*}
	and $\ker(\chi_r - T) = \ker(r - A)$ for every $r \in \mathrm{i}\R$ (see \cite[Theorem IV.3.7 and Corollary IV.3.8]{EN00}).
\end{example}

\begin{example}
	\label{exa:klein-four}
	Consider the $2\times 2$-matrices 
	\begin{align*}
		I \coloneqq \left(\begin{array}{cc}
		1 & 0 \\ 
		0 & 1
		\end{array}\right) \quad \textrm{ and } \quad P \coloneqq \left(\begin{array}{cc}
		0 & 1 \\ 
		1 & 0
		\end{array}
		\right).
	\end{align*}
	If $S$ is the abelian multiplicative semigroup of linear operators on $\C^4$ defined by the permutation matrices
	\begin{align*}
		\left\{\left(\begin{array}{cc}
		I & 0 \\ 
		0 & I
		\end{array}\right), \left(\begin{array}{cc}
		P & 0 \\ 
		0 & I
		\end{array}\right),\left(\begin{array}{cc}
		I & 0 \\ 
		0 & P
		\end{array}\right), \left(\begin{array}{cc}
		P & 0 \\ 
		0 & P
		\end{array}\right)\right\},
	\end{align*}
	then $S$ is isomorphic to the Klein four group. 
	In particular, $\unitaryDual{S}$ has exactly four elements, 
	namely the semigroup characters $\mathds{1}_S$, $\chi$, $\tau$, $\det$ whose values are given in the following table:
	\begin{center}
		\begin{tabular}{c|cccc}
			& 
			$
				\left(\begin{array}{cc}
					I & 0 \\ 
					0 & I
				\end{array}\right)
			$
			&
			$
				\left(\begin{array}{cc}
					P & 0 \\ 
					0 & I
				\end{array}\right)
			$ 
			&
			$
				\left(\begin{array}{cc}
					I & 0 \\ 
					0 & P
				\end{array}\right)
			$ 
			&
			$
				\left(\begin{array}{cc}
					P & 0 \\ 
					0 & P
				\end{array}\right)
			$ 
			\\ 
			\hline 
			$\mathds{1}_S$ & $\phantom{-}1$ & $\phantom{-}1$ & $\phantom{-}1$ & $\phantom{-}1$ \\ 
			$\chi$         & $\phantom{-}1$ & $-1$ & $\phantom{-}1$ & $-1$ \\ 
			$\tau$         & $\phantom{-}1$ & $\phantom{-}1$ & $-1$ & $-1$ \\ 
			$\det$         & $\phantom{-}1$ & $-1$ & $-1$ & $\phantom{-}1$ 
		\end{tabular}
	\end{center}
	Consider the identity map $T = \mathrm{id} \colon S \rightarrow \mathscr{L}(\C^4)$ as a bounded representation of $S$. 
	We claim that $\sigma_\unitary(T) = \sigma_\unitaryPoint(T) = \{\mathds{1}_S, \chi, \tau\}$.
	
	Indeed, one can readily check that 
	$\{\mathds{1}_S, \chi, \tau\} \subseteq \sigma_\unitaryPoint(T) \subseteq \sigma_\unitary(T)$. 
	On the other hand, observe that
	\begin{align*}
		\left(\begin{array}{cc}
			I & 0 \\ 
			0 & I
		\end{array}\right) 
		+ 
		\left(\begin{array}{cc}
			P & 0 \\ 
			0 & P
		\end{array}\right) 
		= 
		\left(\begin{array}{cc}
			P & 0 \\ 
			0 & I
		\end{array}\right) 
		+ 
		\left(\begin{array}{cc}
			I & 0 \\ 
			0 & P
		\end{array}\right)
		,
	\end{align*}
	but 
	\begin{align*}
		\det
		\left(\begin{array}{cc}
			I & 0 \\ 
			0 & I
		\end{array}\right) 
		+  
		\det
		\left(\begin{array}{cc}
			P & 0 \\ 
			0 & P
		\end{array}\right) 
		= 
		2 
		\neq 
		-2 
		= 
		\det
		\left(\begin{array}{cc}
			P & 0 \\ 
			0 & I
		\end{array}\right) 
		+ 
		\det
		\left(\begin{array}{cc}
			I & 0 \\ 
			0 & P
		\end{array}\right)
		,
	\end{align*}
	so the mapping $\det$ cannot be extended to an algebra character in $\Gamma_{\mathrm{A}(T(S))}$ 
	and thus, it follows from \cref{homeomorph} that $\det \notin \sigma_\unitary(T)$
	 
	Observe however, that $\det(A) \in \sigma_{\mathrm{p}}(A) = \sigma(A)$ for every $A \in S$, 
	so this is another counterexample to the question mentioned before \cref{exa:functions-on-circle}.
\end{example}

In the spectral theory of operators, the poles of the resolvent play an important role. 
Recall here that for a power-bounded operator $R$ on a Banach space $E$ a number $\lambda_0 \in \T$ is a \notion{pole of the resolvent} 
if there is a neighborhood $U$ of $\lambda$ such that the pointed neighbourhood $U\setminus\{\lambda\}$ 
is contained in the resolvent set $\C \setminus \sigma(R)$ 
and such that the limit $\lim_{\mu \to \lambda}(\mu-\lambda)(\mu - R)^{-1}$ exists in $\mathscr{L}(E)$; 
see \cite[Definition~VII.3.15]{DuSc1966}. 
(Actually, our definition only covers poles of order at most one. 
However, by \cite[Lemma~VIII.8.1]{DuSc1966} there are no poles of higher order in $\T$ since $T$ is power-bounded.)  
In the situation of semigroup representations, we do not have a resolvent map at our disposal. 
We therefore use the following characterization of poles, based on spectral decomposition, 
as the basis for a generalization (this is an easy consequence of \cite[Theorem~VII.3.18]{DuSc1966}): 
For a power-bounded operator $R \in \mathscr{L}(E)$ on a Banach space $E$ a number $\lambda \in \T$ is a pole of the resolvent 
if and only if there is a projection $P \in \mathscr{L}(E)$ onto $\ker(\lambda -R)$ 
that satisfies $PR = RP$ and $1\notin \sigma_{\unitary}(R|_{\ker(P)})$.

\begin{definition}
	\label{def:pole-riesz}
	Let $E$ be Banach space and let $T \colon S \rightarrow \mathscr{L}(E)$ be a bounded representation. 
	A semigroup character $\chi \in \unitaryDual{S}$ is a \notion{pole} of $T$ 
	if there is a projection $P \in \mathscr{L}(E)$ onto $\ker(\chi - T)$ 
	such that $PT_s=T_sP$ for all $s \in S$ and $\chi \notin \sigma_{\unitary}(T|_{\mathrm{ker}(P)})$. 
	It is a \notion{Riesz point} of $T$ if, in addition, $\dim \ker(\chi - T) < \infty$.
\end{definition}

Essentially the same definition of a pole is given in \cite[Definition~3.1]{Huang95}.
Note that if $\chi, \tau \in \unitaryDual{S}$ and $\chi$ is a pole or a Riesz point of $T$, 
then $\tau \cdot \chi$ is a pole or a Riesz point of $\tau T$ with the same projection $P$. 
This readily follows from \cref{def:pole-riesz} and \cref{compatibility}\ref{compatibility3}.
It follows from \cref{charunierg}\ref{charuniergii} that the projection $P$ in \cref{def:pole-riesz} is uniquely determined.

Poles and Riesz points will be essential for characterizing uniformly mean ergodic 
and quasi-compact bounded semigroup representations in Section~\ref{secuniform} below.
Let us point out, though, that care is needed regarding the relation to $C_0$-semigroup theory: 
if $S = \big([0,\infty),+\big)$ and $T \colon S \to \mathscr{L}(E)$ is a bounded and strongly continuous representation 
-- i.e., a bounded $C_0$-semigroup -- then poles of the resolvent of the generator of $T$ 
that are located in the imaginary axis 
do not, in general, translate to poles in the sense of Definition~\ref{def:pole-riesz}; 
see \cref{rem:rotation-sg} below for details and for a counterexample.
The situation is, however, better for $C_0$-semigroups that are eventually norm continuous:

\begin{example}\label{exppole}
	Consider, as in \cref{examplesspec}\ref{examplesspec1}, 
	a bounded strongly continuous representation $T \colon \R_{\geq 0} \rightarrow \mathscr{L}(E)$ which is eventually norm continuous. 
	Let $\sigma_{\mathrm{pole}}(A)$ denote the set of poles of its generator $A$ within the spectrum $\sigma(A)$  
	and let $\sigma_{\unitary,\mathrm{pole}}(T)$ denote the set of poles of $T$ within the unitary spectrum $\sigma_{\unitary}(T)$.
	Then $\sigma_{\mathrm{pole}}(A) \cap \mathrm{i}\R$ is a finite set by \cite[Theorem~II.4.18]{EN00} 
	and the poles in this set are of order one by the boundedness of $T$ and \cite[Theorem IV.3.6]{EN00}. 
	
	We use the same notation as in \cref{examplesspec}\ref{examplesspec1}: 
	for every $r \in \mathrm{i}\R$ let the unitary semigroup character $\chi_r \in \unitaryDual{S}$ be defined by 
	$\chi_r(t) = \exp(rt)$ for all $t \geq 0$. 
	Then the map 
	\begin{align*}
		\sigma_{\mathrm{pole}}(A) \cap \mathrm{i}\R \rightarrow \sigma_{\unitary,\mathrm{pole}}(T), \quad r \mapsto \chi_r
	\end{align*}
	is a bijection. 
	To see this, use the homeomorphism from \cref{examplesspec}\ref{examplesspec1} 
	and the spectral decomposition of the generator and the semigroup (see \cite[Section A-III.3]{Nage1986}).
	It follows from Example \ref{exa:C_0-pnt-spec} that this bijection maps Riesz points to Riesz points.
\end{example}

The following lemma is an immediate consequence of the compactness of the unitary spectrum (\cref{homeomorph}) 
and the compatibility with sums of representations (\cref{compatibility}).

\begin{lemma}\label{isolated}
	Let $E$ be a Banach space and let $T \colon S \rightarrow \mathscr{L}(E)$ be a bounded representation. 
	If a semigroup character $\chi \in \unitaryDual{S}$ is a pole of $T$, 
	then either $\chi \notin \sigma_{\unitary}(T)$ or 
	$\chi$ is an isolated point of $\sigma_{\unitary}(T)$.
\end{lemma}

\subsection{Representations on ultrapowers}
	\label{upsubsec}

In this section, we consider an important tool for the spectral theory of representations: the ultrapower of a Banach space. 
We briefly recall this construction 
(see for instance \cite[Paragraph V.1]{Schaefer1970}, \cite{Hein1980}, and  \cite[Chapter 8]{DiJaTo1995} for more details).

For a non-empty set $J$ fix an ultrafilter $p$ on $J$. 
Furthermore, denote by $\ell^{\infty}(J,E)$ the set of all bounded functions from \color{red} $J$ \color{black} to $E$. 
The set 
\begin{align*}
	c_{0,p} \coloneqq \bigl\{(x_j)_{j \in J} \in \ell^{\infty}(J,E) \mid \lim_{p} \|x_j\| = 0\bigr\}
\end{align*} 
of all functions from $J$ to $E$ which converge to zero along $p$, is a closed subspace of $\ell^{\infty}(J,E)$. 
Therefore, the quotient space
\begin{align*}
	E^{p} \coloneqq \ell^{\infty}(J,E)/c_{0, p}
\end{align*}
equipped with the quotient norm is a Banach space, which is called the \notion{ultrapower} of $E$ along $p$. 
We introduce the following notation.
\begin{enumerate}[(i)]
	\item 
	For every element $(x_j)_{j \in J} \in \ell^{\infty}(J, E)$ we denote the corresponding equivalence class in $E^{p}$ by $(x_j)^p_{j \in J}$. 
	Then $\|(x_j)_{j \in J}^p\| = \lim_{p} \|x_j\|$ for every $(x_j)_{j \in J} \in \ell^\infty(J,E)$.

	\item 
	We write $x^{p}$ for the equivalence class which contains the constant net $(x)_{j \in J}$ for $x \in E$ 
	and observe that $E \rightarrow E^p, \, x \mapsto x^p$ is a linear isometry.

	\item 
	For every $R \in \mathscr{L}(E)$ we obtain a bounded operator $R^p \in \mathscr{L}(E^p)$ 
	by setting $R^p(x_j)_{j \in J}^p \coloneqq (Rx_j)_{j \in J}^p$ for every $(x_j)_{j \in J}^p\in E^p$. 
	The map $\mathscr{L}(E) \rightarrow \mathscr{L}(E^p), \, R \mapsto R^p$ is then an isometric unital algebra homomorphism.
	
	\item 
	For every representation $T \colon S \rightarrow \mathscr{L}(E)$ we write $T^p \colon S \rightarrow \mathscr{L}(E^p)$ 
	for the induced representation on $E^p$ defined by $(T^p)_s \coloneqq (T_s)^p$ for all $s \in S$.
\end{enumerate}

We prove the following version of a result of Groh (\cite[Proposition~2.2]{Groh1983};  
see also \cite[Proposition~5.2]{Wolff1984}) for bounded group representations.

\begin{theorem}\label{thm:groh}
	Let $E$ be a Banach space.
	For a bounded representation $T \colon S \rightarrow \mathscr{L}(E)$ and a semigroup character $\chi \in \unitaryDual{S}$ 
	the following assertions are equivalent:
	\begin{enumerate}[(a)]
		\item\label{thm:groha} 
		$\chi$ is a Riesz point of $T$.
		
		\item\label{thm:grohb} 
		$\dim \ker(\chi - T^{p}) < \infty$ for each set $J \not= \emptyset$ and each ultrafilter $p$ on $J$.
	\end{enumerate}
	If~\ref{thm:groha} and~\ref{thm:grohb} hold, then $\ker(\chi - T^p) = \{x^p \mid x \in \ker(\chi - T)\}$ 
	for each ultrafilter $p$ on each set $J \not= \emptyset$.
\end{theorem}

To prepare the proof we first show several auxiliary results.

\begin{lemma}\label{fixedspaces}
	Let $p$ be an ultrafilter on a set $J \not= \emptyset$.  
	Let $E$ be a Banach space 
	and $T \colon S \rightarrow \mathscr{L}(E)$ be a bounded representation. 
	If a semigroup character $\chi \in \unitaryDual{S}$ is a pole of $T$, then $\ker(\chi - T^p) = (\ker(\chi - T))^p$.
\end{lemma}

\begin{proof}
	By the definition of a pole of $T$ (\cref{def:pole-riesz}) there exists 
	a projection $P \in \mathscr{L}(E)$ onto $\ker(\chi -T)$ 
	such that $PT_s=T_sP$ for all $s\in S$ and $\chi \notin \sigma_{\unitary}(T|_{\ker(P)})$. 
	
	The inclusion $(\ker(\chi - T))^p \subseteq \ker(\chi - T^p)$ is straightforward. 
	To prove the converse inclusion, take $(x_j)_{j \in J} \in \ell^{\infty}(J,E)$ with $(x_j)_{j \in J}^p \in \ker(\chi - T^p)$. 
	We define $y_j \coloneqq (\mathrm{Id}_E-P)x_j$ for every $j \in J$. 
	If $\lim_{p }\Vert y_j\Vert=0$, then $(x_j)_{j \in J}^p \in (\ker(\chi - T))^p $ as desired. 
	However, if $\lim_{p }\Vert y_j\Vert>0$, then $\chi \in \sigma_{\unitary}(T|_{\ker(P)})$ since
	\begin{align*}
		\lim_{p}\Vert (T|_{\mathrm{kern}(P)})_sy_j-\chi(s)y_j\Vert 
		\leq 
		\|\mathrm{Id}_E - P\| \cdot \lim_{p}\Vert T_sx_j-\chi(s)x_j\Vert 
		= 
		0.
	\end{align*}
	This contradicts the assumption that $\chi$ is a pole of $T$.
\end{proof}

\begin{lemma}\label{invarproj}
	Let $E$ be a Banach space 
	and $T \colon S \rightarrow \mathscr{L}(E)$ a bounded representation with $\dim \fix(T)<\infty$. 
	Then there is a projection $P \in \mathscr{L}(E)$ onto $\fix(T)$ that satisfies $PT_s =T_sP = P$ for all $s \in S$.
\end{lemma}

\begin{proof}
	Pick a basis $x_1, \dots, x_n$ of $\fix(T)$. 
	By the Hahn-Banach theorem we find $x_1', \dots, x_n' \in E'$ with $x_k'(x_\ell) = \delta_{k\ell}$ for all $k,\ell \in \{1, \dots, n\}$. 
	Now take an ergodic net $(A_j)_{j \in J}$ for \color{red} $T'$ \color{black} (see \cref{ergodicnet}). 
	For every $k \in \{1, \dots, n\}$ the net $(A_j x_k')_{j \in J}$ is bounded 
	and thus has, by the Banach-Alaoglu theorem, a subnet that is weak* convergent to a point $y_k' \in E'$. 
	The map
	\begin{align*}
		P \colon E \rightarrow \fix(T), \quad x \mapsto \sum_{k=1}^n \langle y_k', x \rangle x_k
	\end{align*}
	is a bounded projection with range $\fix(T)$ and thus satisfies $T_s P = P$ for every $s \in S$. 
	Since $(A_j)_{j \in J}$ is an ergodic net, one can readily check that $y_k' \in \fix(T')$ for every $k \in \{1, \dots, n\}$. 
	Hence, one also has $P T_s = P$ for every $s \in S$.
\end{proof}

\begin{lemma}\label{nets}
	Let $(x_j)_{j \in J}$ be a net in a complete metric space $(X,d)$ 
	and assume that for every $\varepsilon > 0$ there is a non-empty finite subset $F \subseteq X$ 
	and an index $j_0 \in J$ such that
	\begin{align*}
		\min_{x \in F} d(x_j,x)\leq \varepsilon \textrm{ for every } j \geq j_0.
	\end{align*}
	Then $(x_j)_{j \in J}$ has a subnet that converges in $X$.
\end{lemma}

\begin{proof}
	After passing to a subnet we may assume that $(x_j)_{j \in J}$ is a universal net,  
	which means that for every subset $A \subseteq X$ the net $(x_j)_{j \in J}$ is eventually contained in $A$ 
	or is eventually contained in $X\setminus A$. 
	Now take $\varepsilon > 0$. 
	We choose $F \subseteq E$ finite and $j_0 \in J$ with 
	\begin{align*}
		x_j \in \bigcup_{x \in F} \{y \in X \mid d(x,y) \leq \varepsilon\} \textrm{ for every } j \geq j_0.
	\end{align*}
	Since $(x_j)_{j \in J}$ is universal, 
	we then find $x \in F$ and $j_1 \in I$ such that $d(x_j,x) \leq \varepsilon$ for all $j \geq j_1$ 
	and hence $d(x_j,x_k) \leq 2\varepsilon$ for all $j,k \geq j_1$. 
	This shows that $(x_j)_{j \in J}$ is a Cauchy net and thus converges in $X$.
\end{proof}

\begin{lemma}\label{subnetapprox}
	Let $E$ be a Banach space 
	and let $T \colon S \rightarrow \mathscr{L}(E)$ be a bounded representation, 
	let $\mathds{1}_S \in \sigma_{\unitary}(T)$, 
	and let $(x_j)_{j \in J}$ be an approximate eigenvector of $T$ for $\mathds{1}_S$. 
	Let $p$ be an ultrafilter on $J$ that contains all the tails $\{j \in J \mid j \geq j_0\}$ for $j_0 \in J$. 
	If $\dim \fix(T^p) < \infty$,
	then $(x_j)_{j \in J}$ has a subnet that converges to a point in $\fix(T)$. 
\end{lemma}

\begin{proof}
	Let $(x_j)_{j \in J}$ be an approximate eigenvector of $T$ for the spectral element $\mathds{1}_S$. 
	If $(x_j)_{j \in J}$ has a convergent subnet, then the limit is clearly in $\fix(T)$, 
	so assume now for a contradiction that $(x_j)_{j \in J}$ does not have a convergent subnet. 
	By \cref{nets} we then find an $\varepsilon > 0$ such that for every non-empty finite subset $F \subseteq E$ and every $j_0 \in J$ 
	there is some $j \geq j_0$ that satisfies $\min_{x \in F} \|x_j - x\| \geq  \varepsilon$.
	Now we recursively construct a sequence $(x^n)_{n \in \N} = ((x^n_j)_{j \in J})_{n \in \N}$ in $\ell^\infty(J,E)$ 
	that satisfies the following properties: 
	\begin{enumerate}[(i)]
		\item 
		$(x^n)^p \in \fix(T^p)$ with $\|(x^n)^p\| = 1$ for every $n \in \N$;
		
		\item 
		$\|(x^n)^p - (x^m)^p\| \geq \varepsilon$ for all distinct $m,n \in \N$.
	\end{enumerate}
	As the unit sphere of the finite-dimensional space $\fix(T^p)$ is totally bounded, this yields the desired contradiction.
	
	For $n=0$ we take $x^0 \coloneqq (x_j)_{j \in J}$; 
	note that $(x^0)^p \in \fix(T^p)$ by the choice of the ultrafilter $p$.
	Now assume that $x^0, \dots, x^n \in \ell^\infty(J,E)$ have already been constructed for some $n \in \N$. 
	For every $j \in J$ we find some $\phi(j) \in J$ with $\phi(j) \geq j$ such that
	\begin{align*}
		\|x_{\phi(j)} - x_j^k\| \geq \varepsilon \textrm{ for every } k \in \{1, \dots, n\}.
	\end{align*}
	Then, $x^{n+1}\coloneqq (x_{\varphi(j)})_{j \in J}$ is a subnet of $(x_j)_{j \in J}$ 
	and thus also an approximate eigenvector for $T$ with respect to $\mathds{1}_S$. 
	Hence, $(x^{n+1})^p \in \fix(T^p)$ by the choice of the ultrafilter $p$. 
	Moreover, one clearly has $\|(x^{n+1})^p\| = 1$. 
	For all $k \in \{1, \dots, n\}$ the construction of $x^{n+1}$ gives $\|x^{n+1}_j - x_j^k\| \geq \varepsilon$  
	for all $j \in J$ and thus $\|(x^{n+1})^p - (x^k)^p\| \geq \varepsilon$. 
\end{proof}

We finally prove \cref{thm:groh}.

\begin{proof}[Proof of \cref{thm:groh}.]
	Since $\ker(\chi - T)=\mathrm{fix}(\overline{\chi} T)$, 
	\cref{compatibility}\ref{compatibility3} shows that is suffices to prove the result for $\chi = \mathds{1}_S$.

	\ImpliesProof{thm:groha}{thm:grohb} 
	Let $\mathds{1}_S$ be a Riesz point of $T$, 
	let $J \not= \emptyset$, and let $p$ be an ultrafilter on $J$.
	Since $\mathds{1}_S$ is a pole, \cref{fixedspaces} gives $\mathrm{fix}(T^{p}) = \mathrm{fix}(T)^{p}$. 
	As $\mathrm{fix}(T)$ is finite-dimensional, it is isomorphic to $\mathrm{fix}(T)^{p}$ 
	and thus, the latter space is finite-dimensional, too. 
	
	\ImpliesProof{thm:grohb}{thm:groha}
	Assumption~\ref{thm:grohb} implies, in particular, that $\mathrm{fix}(T)$ is finite-dimensional. 
	By \cref{invarproj} we thus find a projection $P \in \mathscr{L}(E)$ onto $\fix(T)$ satisfying $PT_s = T_sP = P$ for all $s \in S$. 
	We show that $\mathds{1}_S \notin \sigma_{\unitary}(T|_{\ker(P)})$ which proves that $\mathds{1}_S$ is a Riesz point of $T$. 
	
	Assuming the contrary, we find an approximate eigenvector $(x_j)_{j \in J_S}$ of $T|_{\ker(P)}$ with respect to $\mathds{1}_S$. 
	Let $p$ be an ultrafilter on $J$ that contains all the tails $\{j \in J \mid j \geq j_0\}$ for $j_0 \in J$.
	Since $\dim \mathrm{fix}(T^{p}) < \infty$ according to~\ref{thm:grohb}, 
	\cref{subnetapprox} shows that $(x_j)_{j \in J}$ has a subnet that converges to a point $x\in \ker(P) \cap \fix(T)$. 
	Clearly, $x \neq 0$, which contradicts the properties of $P$.
\end{proof}

\section{Uniform ergodic theorems}\label{secuniform}

Now we apply the techniques developed so far 
to generalize classical uniform ergodic theorems for power-bounded operators to bounded repesentations of abelian semigroups.

\subsection{Uniformly mean ergodic representations}\label{subsec:erg}

We start with the concept of uniformly mean ergodic semigroups, 
which is discussed in detail in \cite[Chapter~W]{DNP1987}, \cite[Section~V.4]{EN00} and \cite[Paragraph~2.2]{Kren1985}).

\begin{definition}
	\label{def:unif-mean-ergodic}
	Let $E$ be a Banach space. 
	A bounded representation $T \colon S \to \mathscr{L}(E)$ is \notion{uniformly mean ergodic} 
	if the operator norm closed convex hull $\overline{\mathrm{co}}\, T(S)\subseteq \mathscr{L}(E)$ has a zero element $P$, 
	meaning that $RP=PR=P$ for every $R \in \overline{\mathrm{co}}\, T(S)$. 
	In this case the zero element $P$ (which is clearly unique) is called the \notion{mean ergodic projection} of $T$.
\end{definition}

In the case where $S = \big([0,\infty),+\big)$ and $T \colon S \to \mathscr{L}(E)$ is a $C_0$-semigroup, 
one has to be extremely careful not to confuse uniform mean ergodicity as given by Definition~\ref{def:unif-mean-ergodic} 
with the common definition of uniform mean ergodicity in $C_0$-semigroup theory; 
we discuss this distinction in detail in Remark~\ref{rem:rotation-sg} below. 

For $S = (\N,+)$ and $(T_1^s)_{s \in \N}$ though, there is no such problem 
and uniform mean ergodicity as defined in Definition~\ref{def:unif-mean-ergodic} can be checked to be equivalent 
to the common definition that the Cesàro means $\frac{1}{n} \sum_{k=0}^{n-1} T_1^k$ converge with respect to the operator norm 
as $n \to \infty$.

\begin{lemma}\label{lemmeanergproj}
	Let $E$ be a Banach space and let $T \colon S \to \mathscr{L}(E)$ be a bounded representation.
	If $T$ is uniformly mean ergodic, then its mean ergodic projection $P$ 
	is a projection onto the fixed space of $T$.
\end{lemma}

\begin{proof}
	Since $P$ is a zero element of $\overline{\mathrm{co}}\, T(S)$, we have $P^2=P$. 
	Moreover, $T_s P = P$ for all $s \in S$ implies that $\mathrm{rg}(P) \subseteq \fix(T)$ 
	and conversely, $P \in \overline{\mathrm{co}}\,T(S)$ implies $\mathrm{fix}(T) \subseteq \mathrm{rg}(P)$.
\end{proof}

The following result provides several characterizations of uniformly mean ergodic representations; 
we refer to \cite[Theorem~W.3]{DNP1987} and \cite[Paragraph 2.2]{Kren1985} for similar characterizations in the case of single operators. 
For a Banach space $E$, a representation $T \colon S \rightarrow \mathscr{L}(E)$, and a semigroup character $\chi \in \dual{S}$ we write
\begin{align*}
	\mathrm{rg}(\chi - T) 
	\coloneqq 
	\sum_{s \in S} \mathrm{rg}(\chi(s) - T_s)
	,
\end{align*}
where the sum denotes the subspace that consist of all finite sums of elements of the spaces $\mathrm{rg}(\chi(s) - T_s)$ for $s \in S$.

\begin{theorem}\label{charunierg}
	Let $E$ be a Banach space. 
	For a bounded representation $T \colon S \rightarrow \mathscr{L}(E)$ the following conditions are equivalent.
	\begin{enumerate}[(a)]
		\item\label{charunierga} 
		$T$ is uniformly mean ergodic.
		
		\item\label{charuniergb} 
		There exists an ergodic net for $T$ that converges in operator norm.
	
		\item\label{charuniergc} 
		Every ergodic net for $T$ converges in operator norm.
		
		\item\label{charuniergd} 
		The semigroup character $\mathds{1}_S \in \unitaryDual{S}$ is a pole of $T$.
		
		\item\label{charunierge} 
		$\mathrm{rg}(\mathds{1}_S - T)$ is closed in $E$ 
		and there are $s_1, \dots, s_n \in S$ such that
		\begin{align*}
			\mathrm{rg}(\mathds{1}_S - T) = \sum_{k=1}^n \mathrm{rg}(1 - T_{s_k})
		\end{align*}
	\end{enumerate}
	If the equivalent conditions~{\upshape\ref{charunierga}}--{\upshape\ref{charunierge}} are satisfied, 
	then the following assertions hold.
	\begin{enumerate}[(i)]
		\item\label{charuniergi} 
		Every ergodic net for $T$ converges in operator norm to the mean ergodic projection.
		
		\item\label{charuniergii} 
		The projection $P \in \mathscr{L}(E)$ that exists according to \cref{def:pole-riesz} 
		since $\mathds{1}_S \in \unitaryDual{S}$ is a pole, 
		is uniquely determined and coincides with the mean ergodic projection. 
		
		\item\label{charuniergiii} 
		The space $E$ can be decomposed as $E = \fix T \oplus \mathrm{rg}(\mathds{1}_S - T)$.
	\end{enumerate}
\end{theorem}

For the proof we use the following lemma which is a direct consequence 
of \cref{prop:charspec}\ref{prop:charspec:itm:ideal-right}.

\begin{lemma}\label{rangesum}
	Let $E$ be a Banach space, let $T \colon S \rightarrow \mathscr{L}(E)$ be a bounded representation, 
	and assume that $\mathds{1}_S \notin \sigma_{\unitary}(T)$. 
	Then there are $s_1, \dots, s_n \in S$ with
	\begin{align*}
		E = \sum_{k=1}^n \mathrm{rg}(1 - T_{s_k}). 
	\end{align*}
\end{lemma}

\begin{proof}[Proof of \cref{charunierg}.]
	The equivalence of assertions \ref{charunierga}, \ref{charuniergb} and \ref{charuniergc} is already known 
	(see, e.g., \cite[Theorem 2.2]{Kren1985}), but we include short arguments for them for the reader's convenience. 
	
	\ImpliesProof{charuniergc}{charuniergb}
	This implication holds since there exists an ergodic net for $T$ according to \cref{ergodicnet}.
	
	\ImpliesProof{charuniergb}{charunierga} 
	This follows from the fact that the limit of a convergent ergodic net 
	is a zero element for $\overline{\mathrm{co}}\, T(S)$, which is easy to check.
	
	\ImpliesProof{charunierga}{charuniergc}
	Let $P \in \overline{\mathrm{co}}\, T(S)$ be the mean ergodic projection and assume that $(C_j)_{j \in J}$ is an ergodic net for $T$. 
	Since $\lim_{j}(1 - T_s)C_j \rightarrow 0$ for every $s \in S$, 
	we also obtain $\lim_{j} (1 - R)C_j = 0$ for every $R \in \overline{\mathrm{co}}\, T(S)$ by approximation. 
	In particular, $\lim_{j} (1 - P)C_j = 0$. 
	On the other hand, $PC_j = P$ for every $j \in J$, and thus $\lim_{j} C_j = P$. 
	
	\ImpliesProof{charunierga}{charuniergd}
	Let $T$ be uniformly mean ergodic with mean ergodic projection $P \in \overline{\mathrm{co}}\, T(S)$. 
	We show that $\mathds{1}_S \notin \sigma_{\unitary}(T|_{\ker(P)})$ which implies that $\mathds{1}_S$ is a pole of $T$. 
	Assume conversely that $\mathds{1}_S \in \sigma(T|_{\ker(P)})$. 
	According to \cref{homeomorph} we can thus find an algebra character $\psi \in \Gamma_{\mathrm{A}(T|_{\ker(P)}(S))}$ 
	that satisfies $\psi(T_s|_{\ker(P)}) = 1$ for all $s \in S$. 
	Then also $\psi(R|_{\ker(P)}) = 1$ for every $R \in \overline{\mathrm{co}}\, T(S)$, 
	hence $\psi(0_{\ker(P)}) = \psi(P|_{\ker(P)}) =1$, a contradiction.

	\ImpliesProof{charuniergd}{charunierge} 
	Let $\mathds{1}_S \in \unitaryDual{S}$ be a pole of $T$, 
	which means that there exists a projection $P \in \mathscr{L}(E)$ onto $\fix(T)$ 
	that satisfies $PT_s = T_sP$ for all $s \in S$ and $\mathds{1}_S \notin \sigma_{\unitary}(T|_{\ker(P)})$. 
	By \cref{rangesum} we then find $s_1, \dots, s_n \in S$ 
	with $\ker(P) \subseteq \sum_{k=1}^n \mathrm{rg}(1-T_{s_k}) \subseteq \mathrm{rg}(\mathds{1}_S - T) $. 
	On the other hand, we have
	\begin{align*}
		P(1 - T_s) = P - PT_s = 0
	\end{align*}
	for each $s \in S$, hence $\mathrm{rg}(\mathds{1}_S - T) \subseteq \ker(P)$. 
	Consequently, $\ker(P) = \mathrm{rg}(\mathds{1}_S - T) = \sum_{k=1}^n\mathrm{rg}(1-T_{s_k})$. 
	This shows~\ref{charunierge} since $\ker P$ is closed. 
	
	\ImpliesProof{charunierge}{charuniergc}
	Assume that $E_0 \coloneqq \mathrm{rg}(\mathds{1}_S - T)$ is closed 
	and that there are $s_1, \dots, s_n \in S$ with $E_0 = \sum_{k=1}^n \mathrm{rg}(1 - T_{s_k})$. 
	Let $(C_j)_{j \in J}$ be an ergodic net for $T$. 
	The surjective bounded linear map
	\begin{align*}
		E^n \rightarrow E_0,\quad (x_1, \dots, x_n) \mapsto \sum_{k=1}^n (1-T_{s_k})x_j
	\end{align*}
	is open by the open mapping theorem where the norm on $E^n$ 
	is given by $\|(x_1, \dots, x_n)\| \coloneqq \max_{j =1, \dots, n} \|x_j\|$ for $(x_1, \dots, x_n) \in E^n$. 
	Thus, we can choose $c >0$ such that for every $y \in E_0$ of norm $\|y\| \le 1$
	we find $(x_1, \dots, x_n)$ such that $y=\sum_{k=1}^n(1-T_{s_k})x_k$ and $\max_{k=1, \dots, n} \|x_k\| \leq c$. 
	For each such $y$ this gives
	\begin{align*}
		\|C_j y\| 
		= 
		\left\|\sum_{k=1}^n C_j(1 - T_{s_k})x_k\right\| 
		\leq 
		c n \; \max_{k=1, \dots,n}\| C_j(1 - T_{s_k})\| 
	\end{align*}
	for every $j \in J$.
	Thus, $\|C_j|_{E_0}\| \leq c n \max_{k=1,\dots, n}\| C_j(1 - T_{s_k})\|$ for each $j \in J$ 
	and therefore $\lim_j {C_j}|_{E_0} = 0$. 
	This implies that $\fix(T|_{E_0}) = 0$ and, 
	by the equivalence of \ref{charunierga} and \ref{charuniergb}, 
	it also gives that the restricted representation $T|_{E_0}$ is uniformly mean ergodic. 
	Hence, we conclude from the implication ``\ref{charunierga} $\Rightarrow$ \ref{charuniergd}'' proved above 
	that $\mathds{1}_S \notin \sigma_{\unitary}(T|_{E_0})$. 
	Thus, using \cref{prop:charspec}\ref{prop:charspec:itm:ideal} 
	we find operators $R_1, \dots, R_m \in \mathrm{A}(T|_{E_{0}}(S)) \subseteq \mathscr{L}(E_0)$ 
	as well as $s_1, \dots, s_m \in S$ such that  
	\begin{align*}
		\mathrm{Id}_{E_0} = \sum_{k=1}^m R_k (1 - T_{s_k}|_{E_0}).
	\end{align*}
	Hence, the operator $Q \coloneqq \sum_{k=1}^m R_k (1 - T_{s_k}) \in \mathscr{L}(E)$ 
	-- which is well-defined and maps into $E_0$ as all the $1 - T_{s_k}$ map into $E_0$ -- 
	is a projection with range $E_0$. 
	For each $j \in J$ one has $\mathrm{rg}(1 - C_j) \subseteq E_0$ since $E_0$ is closed,  
	and thus $(1 - C_j) = Q(1 - C_j)$.
	This yields
	\begin{align*}
		C_j - (1 - Q) = Q C_j  = \sum_{k=1}^m R_k (1 - T_{s_k})C_j \rightarrow 0.
	\end{align*}
	so~\ref{charuniergc} holds. 
	
	Now assume that the equivalent conditions~\ref{charunierga}--\ref{charunierge} hold. 
	We prove~\ref{charuniergi},~\ref{charuniergii} and \ref{charuniergiii}.
	
	\ref{charuniergi}\color{red}$\colon$\color{black}
	This follows from the proof of the implication~``\ref{charunierga} $\Rightarrow$ \ref{charuniergc}'' above. 
	
	\ref{charuniergii}~and~\ref{charuniergiii}\color{red}$\colon$\color{black}
	Let $P$ be as in Definition~\ref{def:pole-riesz}, 
	i.e., $P$ is a projection onto $\ker(\one_S - T) = \fix(T)$, 
	it commutes with all the operators $T_s$ for $s \in S$, 
	and $\one_S \not\in \sigma_{\unitary}(T|_{\ker P})$. 
	The proof of the implication~``\ref{charuniergd} $\Rightarrow$ \ref{charunierge}'' 
	shows that $\ker P = \mathrm{rg}(\mathds{1}_S - T)$, 
	which proves~\ref{charuniergiii}. 
	
	Now take an ergodic net $(C_j)_{j \in J}$ for $T$. 
	Then each operator $C_j$ acts as the identity on $\fix T = \mathrm{rg}(P)$. 
	On the other hand, the proof of the implication~``\ref{charunierge} $\Rightarrow$ \ref{charuniergc}'' 
	shows that $(C_j)_{j \in J}$ converges to $0$ in operator norm on $\mathrm{rg}(\mathds{1}_S - T) = \ker P$.
	Hence, $(C_j)_{j \in J}$ converges in operator norm to $P$. 
	But we already know from~\ref{charuniergi} that $(C_j)_{j \in J}$ converges to the mean ergodic projection.
\end{proof}

\begin{corollary}
	\label{cor:conv-mean-ergodic}
	Let $E$ be a Banach space and let $T \colon S \rightarrow \mathscr{L}(E)$ be a bounded representation. 
	If $(T_s)_{s \in S}$ converges in operator norm to a an operator $P \in \mathscr{L}(E)$, 
	then $T$ is uniformly mean ergodic, $P$ is the mean ergodic projection of $T$, 
	and every ergodic net for $T$ converges in operator norm to $P$.
\end{corollary}

\begin{proof}
	It is easy to check that $P$ is a projection and that $P T_s = T_s P = P$ for all $s \in S$. 
	This implies that the assumptions of Definition~\ref{def:unif-mean-ergodic} are satisfied, 
	i.e., $T$ is uniformly mean ergodic with mean ergodic projection $P$. 
	\cref{charunierg}\ref{charuniergi} thus gives that every ergodic net for $T$ converges in operator norm to $P$.
\end{proof}

The following example shows that in the situation of \cref{charunierg} the condition that $\mathrm{rg}(\mathds{1}_S - T)$ is closed 
does not imply uniform mean ergodicity of $T$ -- so the second part of condition~\ref{charunierge} cannot be dropped.

\begin{example}
	\label{exa:intersection-sg}
	Consider an uncountable set $X$ and the set $\mathcal{P}_{\mathrm{coc}}(X)$ of all cocountable subsets of $X$. 
	Equipped with intersection, this becomes an abelian semigroup. 
	Now consider the Banach space $\ell^1(X)$ of all absolutely summable sequences indexed over $X$. 
	For every $M \in \mathcal{P}_{\mathrm{coc}}(X)$ we consider the multiplication operator 
	$T_M \in \mathscr{L}(\ell^1(X))$ defined by $T_Ma \coloneqq \mathbbm{1}_M a$ 
	for all $a \in \ell^1(X)$. 
	Then $\mathcal{P}_{\mathrm{coc}}(X) \rightarrow \mathscr{L}(\ell^1(X)), \, M \mapsto T_M$ 
	is a bounded representation of $\mathcal{P}_{\mathrm{coc}}(X)$. 
	Moreover, $\mathrm{rg}(\mathds{1}_S - T) = \ell^1(X)$ since every $a \in  \ell^1(X)$ vanishes outside of a countable subset of $X$. 
	In particular, $\mathrm{rg}(\mathds{1}_S - T)$ is closed. 
	However, $\sum_{k=1}^n \mathrm{rg}(1 - T_{M_k}) \neq \ell^1(X)$ 
	for all $M_1, \dots, M_n \in \mathcal{P}_{\mathrm{coc}}(X)$ since $X$ is uncountable.
\end{example}

\subsection{Totally uniformly mean ergodic representations}\label{subsec:total}

For a power-bounded operator $R \in \mathscr{L}(E)$ it is known that $T$ is uniformly almost periodic, 
i.e., the set $\{R^n \mid n \in \N\}$ is a relatively compact subset of $\mathscr{L}(E)$, 
if and only if $T$ is totally uniformly mean ergodic, 
i.e., all rotated operators $\lambda T$ for $\lambda \in \T$ are uniformly mean ergodic 
(see \cite[Lemma~2.12 and Proposition~2.13]{DoGl2021}). 
In Theorem~\ref{chartotallyuniform} below we generalize this result to bounded representations of abelian semigroups.

\begin{definition}
	\label{def:totally-erg}
	Let $E$ be a Banach space. 
	A bounded representation $T \colon S \to \mathscr{L}(E)$ is called \notion{totally uniformly mean ergodic} 
	if $\chi T$ is uniformly mean ergodic for every $\chi \in \unitaryDual{S}$.
\end{definition}

In general we cannot expect that a uniformly mean ergodic representation $T \colon S \rightarrow \mathscr{L}(E)$ 
of an abelian semigroup $S$ has a relatively compact range. 
Instead we use the \notion{semigroup at infinity}.
This notion was introduced in \cite{GlHa2019} for the strong operator topology. 
In \cite{DoGl2021} the same concept was studied for the operator norm topology, in which case it is defined as follows.

\begin{definition}
	\label{def:asymp}
	Let $E$ be a Banach space and let $T \colon S \to \mathscr{L}(E)$ be a represenation. 
	\begin{align*}
		\mathscr{T}_\infty 
		\coloneqq 
		\bigcap_{s_0 \in S} \overline{ \{T_s \mid s \ge s_0\} }
		,
	\end{align*}
	where the closure in taken in the operator norm, is called the \notion{semigroup at infinity} of $T$.
\end{definition}

One can check that the semigroup at infinity consists precisely of the limits of all convergent subnets of $(T_s)_{s \in S}$.
It is easy to see that the semigroup at infinity is a subsemigroup of the multiplicative semigroup $\mathscr{L}(E)$. 

We will also use the notion of a semigroup ideal in the next theorem. 
Recall here that a subset $I$ of a commutative semigroup $(C,+)$ is called 
an \notion{ideal} if $I$ is non-empty and satisfies $c + d \in I$ for all $c \in C$ and $d \in I$.

\begin{theorem}\label{chartotallyuniform}
	Let $E$ be a Banach space.
	For a bounded representation $T \colon S \rightarrow \mathscr{L}(E)$ the following assertions are equivalent.
	\begin{enumerate}[(a)]
		\item\label{chartotallyuniform-sg-infty} 
		The semigroup at infinity $\mathscr{T}_\infty$ of $T$ is non-empty and operator norm compact. 
		
		\item\label{chartotallyuniform-subnet} 
		Every subnet of $(T_s)_{s \in S}$ has a subnet that converges in operator norm. 
		
		\item\label{chartotallyuniform-ideal} 
		The semigroup $\overline{T(S)}$ contains an ideal that is compact in operator norm.
		
		\item\label{chartotallyuniform-constr}
		There exists a non-empty set $C \subseteq \mathscr{L}(E)$ 
		that is relatively compact in operator norm and that satisfies $\operatorname{dist}(T_s,C) \overset{s}{\to} 0$.
		
		\item\label{chartotallyuniform-erg} 
		The representation $T$ is totally uniformly mean ergodic.
		
		\item\label{chartotallyuniform-poles}
		The unitary spectrum $\sigma_{\unitary}(T)$ consists of poles of $T$.
		
		\item\label{chartotallyuniform-decomp} 
		There exists a decomposition $E = E_{\mathrm{r}} \oplus E_{\mathrm{s}}$ of $E$ into closed $T$-invariant subspaces such that
		\begin{itemize}
			\item 
			$E_{\mathrm{r}} = \bigoplus_{k=1}^n \ker(\chi_k - T)$ for some $\chi_1, \dots, \chi_n \in \unitaryDual{S}$, and
			
			\item 
			the net $\big( T|_{E_{\mathrm{s}}} \big)_{s \in S}$ converges to $0$ with respect to the operator norm.
		\end{itemize}
	\end{enumerate}
\end{theorem}

Assertion~\ref{chartotallyuniform-subnet} in the theorem can, alternatively, be formulated 
in terms of universal nets, see \cite[Proposition~2.8(iii)]{DoGl2021}; 
moreover, the assertion is equivalent to what is called \notion{asymptotic compactness} in \cite[Definition~3.2]{Huang95}.
If $S$ contains a majorizing sequence, assertion~\ref{chartotallyuniform-subnet} 
can also be formulated in terms of sequences, see \cite[Proposition~2.8(iv)]{DoGl2021}. 
In the case where $S$ is cancellative, the equivalence of~\ref{chartotallyuniform-subnet} 
and~\ref{chartotallyuniform-poles} in the theorem is a special case of \cite[Proposition~G]{Huang95}.

Assertion~\ref{chartotallyuniform-constr} in the theorem is inspired by the notion of 
\notion{constrictive} semigroups that can be used to study the long-term behaviour of operator semigroups 
in the strong topology. 
We refer to \cite[Section~1.3]{Em07} for a detailed treatment of constrictive semigroups. 
It is, however, worth pointing out explicitly that the fixed space of constrictive semigroups is always finite-dimensional 
(\cite[Theorem~1.3.3]{Em07}), while this need not be the case 
if property~\ref{chartotallyuniform-constr} in \cref{chartotallyuniform} is satisfied
(consider for instance the representation that maps every $s \in S$ to the identity operator 
on an infinite-dimensional space $E$).

\begin{proof}[Proof of \cref{chartotallyuniform}]
	\EquivalentProof{chartotallyuniform-sg-infty}{chartotallyuniform-subnet} 
	This equivalence is proved in \cite[Proposition~2.8(i) and~(ii)]{DoGl2021}.
	
	\ImpliesProof{chartotallyuniform-sg-infty}{chartotallyuniform-ideal} 
	It is straightforward to check that the set $\mathscr{T}_\infty$, 
	which is non-empty and compact by~\ref{chartotallyuniform-sg-infty}, is an ideal in $\overline{T(S)}$.
	
	\ImpliesProof{chartotallyuniform-ideal}{chartotallyuniform-constr}
	Assume that $\overline{T(S)}$ contains a compact ideal $I$ 
	and let $R \in I$. 
	Fix $\varepsilon > 0$ and set $M \coloneqq \sup_{s \in S} \lVert T_s \rVert < \infty$. 
	There exists an element $s_0 \in S$ such that $\lVert T_{s_0} - R \rVert \le \varepsilon$. 
	Hence, for every $r \in S$ one thus has $\lVert T_{s_0+r} - R T_r \rVert \le M \varepsilon$. 
	Since $I$ is an ideal, the operator $R T_r$ is in $I$, so it follows that $\operatorname{dist}(T_s, I) \le M \varepsilon$ 
	for every $s \ge s_0$.
	
	\ImpliesProof{chartotallyuniform-constr}{chartotallyuniform-subnet}
	Let $C \subseteq \mathscr{L}(E)$ be as in~\ref{chartotallyuniform-constr}.
	After replacing $C$ with its closure we may, and shall, assume that $C$ is closed in operator norm.
	Consider a subnet $(T_{s_j})_{j \in J}$ of $(T_s)_{s \in S}$. 
	For every $j \in J$ we can choose an element $C_j \in C$ such that $\lVert T_j - C_j \rVert \le 2 \operatorname{dist}(T_{s_j}, C)$; 
	this is clear if the distance of $T_{s_j}$ to $C$ is non-zero and if the distance is $0$, it follows from the closedness of $C$.
	The net $(C_j)_{j \in J}$ in the compact set $C$ has a convergent subnet $(C_{j_k})_{k \in K}$. 
	Since $\operatorname{dist}(T_{s_j}, C) \to 0$, we conclude that $\big(T_{s_{j_k}}\big)_{k \in K}$ also converges.
	
	\ImpliesProof{chartotallyuniform-sg-infty}{chartotallyuniform-erg}
	Assume that~\ref{chartotallyuniform-sg-infty} holds. 
	We first show that $T$ is uniformly mean ergodic.
	As shown in \cite[Theorem~2.4]{DoGl2021}, it follows from~\ref{chartotallyuniform-sg-infty} 
	that there exists a projection $P_\infty \in \mathscr{T}_\infty$ 
	that commutes with all the operators $T_s$ and that satisfies $\lim_s T_s|_{\ker P_\infty} = 0$ in operator norm. 
	Take an ergodic net $(C_j)_{j \in J}$ of $T$. 
	It follows from \cref{cor:conv-mean-ergodic} that $\lim_j C_j|_{\ker P_\infty} = 0$, 
	so $\lim_j C_j(1-P_\infty) = 0$. 
	On the other hand, the net $\big( C_j P_\infty \big)_{j \in J}$ 
	is contained in the compact ideal $\mathscr{T}_\infty$ of the semigroup $\overline{T(S)}$ 
	and thus has a convergent subnet. 
	Hence, $(C_j)_{j \in J}$ also has a convergent subnet. 
	Since this subnet is an ergodic net, too, it follows 
	from \cref{charunierg}\ref{charunierga} and~\ref{charuniergb} that $T$ is uniformly mean ergodic.
	
	Now let $\chi \in \unitaryDual{S}$. 
	By using the equivalence of~\ref{chartotallyuniform-sg-infty} and~\ref{chartotallyuniform-subnet} 
	one easily checks that $\chi T$ also satisfies~\ref{chartotallyuniform-sg-infty}. 
	Hence, as we have just shown, $\chi T$ is uniformly mean ergodic, too.
	
	\ImpliesProof{chartotallyuniform-erg}{chartotallyuniform-poles}
	This implication follows from \cref{compatibility}\ref{compatibility3} 
	and \cref{charunierg}\ref{charunierga} and~\ref{charuniergd}.
	
	\ImpliesProof{chartotallyuniform-poles}{chartotallyuniform-decomp}
	Assume that the unitary spectrum $\sigma_{\unitary}(T)$ consist of poles only. 
	By \cref{homeomorph}, $\sigma_{\unitary}(T)$ is compact, 
	and each pole is isolated in it due to \cref{isolated}. 
	Hence, $\sigma_{\unitary}(T)$ is finite. 
	We write $P_{\overline{\chi}} \in \overline{\mathrm{co}}(\overline{\chi} T(S)) \subseteq \mathrm{A}(T(S))$ 
	for the mean ergodic projection of $\overline{\chi}T$ for all $\chi \in \sigma_{\unitary}(T)$. 
	Then,
	\begin{align*}
		\mathrm{ran}(P_{\overline{\chi}}P_{\overline{\tau}}) 
		= 
		\mathrm{ran}(P_{\overline{\tau}}P_{\overline{\chi}}) 
		\subseteq 
		\ker(\chi - T) \cap \ker(\tau -T) = \{0\}
	\end{align*}
	for $\chi, \tau \in \sigma_{\unitary}(T)$ with $\chi \neq \tau$. 
	Thus, we obtain a projection $P \coloneqq \sum_{\chi \in \sigma_{\unitary}(T)} P_{\overline{\chi}}$ 
	onto $E_{\mathrm{r}} \coloneqq \bigoplus_{\chi \in \sigma_{\unitary}(T)} \ker(\chi -T)$ with $PT_s=T_sP$ for all $s \in S$. 
	Moreover, $\sigma_{\unitary}(T|_{\ker(P)}) \subseteq \sigma_{\unitary}(T|_{\ker(P_{\overline{\chi}})})$ 
	for every $\chi \in \sigma_{\unitary}(T)$ by \cref{compatibility}\ref{compatibility1} 
	and hence $\sigma_{\unitary}(T|_{\ker(P)}) = \emptyset$. 
	By \cref{stability} we obtain that $0 \in \overline{T|_{\ker(P)}(S)}$, 
	which proves~\ref{chartotallyuniform-decomp} with $E_{\mathrm{s}} \coloneqq \ker(P)$. 
	
	\ImpliesProof{chartotallyuniform-decomp}{chartotallyuniform-subnet}
	Consider a subnet $(T_{s_j})_{j \in J}$. 
	For each $k \in \{1, \dots, n\}$, the representation $T$ acts on $\ker(\chi_k - T)$ as the multiplication with $\chi_k$. 
	Since $\T^n$ is compact, we conclude that there exists 
	a subnet of $(T_{s_j})_{j \in J}$ that converges on $E_{\mathrm{r}}$. 
	Since the same subnet converges to $0$ on $E_{\mathrm{s}}$, assertion~\ref{chartotallyuniform-subnet} holds.
\end{proof}

The equivalent assertions in \cref{chartotallyuniform} are very useful in the following sense: 
if the semigroup at infinity $\mathscr{T}_\infty$ is non-empty and compact 
-- i.e., if the equivalent assertion~\ref{chartotallyuniform-sg-infty} is satisfied -- 
one can use the results from \cite[Sections~3 and~4]{DoGl2021} to derive sufficient conditions 
for the operator norm convergence of the net $(T_s)_{s \in S}$.

\subsection{Quasi-compact representations}\label{subsec:quasi-compact}

We now discuss quasi-compactness of semigroup representations, 
a concept that is closely related to total uniform mean ergodicity. 
We start with the following direct consequence of \cref{stability}. 
Here, for a Banach space $E$ we write $\mathscr{K}(E) \subseteq \mathscr{L}(E)$ for the closed ideal of compact operators on $E$.

\begin{proposition}\label{charquasicomp}
	Let $E$ a Banach space. 
	Let $T \colon S \rightarrow \mathscr{L}(E)$ be a bounded representation 
	and $T_{/\mathscr{K}(E)} \colon S \rightarrow \mathscr{L}(E)/\mathscr{K}(E)$ 
	the induced representation in $\mathscr{L}(E)/\mathscr{K}(E)$. 
	Then the following assertions are equivalent.
	\begin{enumerate}[(a)]
		\item 
		$\sigma_{\unitary}(T_{/\mathscr{K}(E)}) = \emptyset$.
		
		\item 
		There is $s \in S$ and $K \in \mathscr{K}(E)$ with $\|T_s - K \| < 1$.
	\end{enumerate}
\end{proposition}

\begin{definition}
	Let $E$ be a Banach space. 
	A bounded representation $T \colon S \rightarrow \mathscr{L}(E)$ is \notion{quasi-compact} 
	if it satisfies the equivalent conditions of \cref{charquasicomp}.
\end{definition}

Quasi-compact operators are discussed, e.g., in \cite[Section VIII.8]{DuSc1966} and \cite[Chapter~W]{DNP1987}, 
and quasi-compact strongly continuous one-parameter semigroups in \cite[Section V.3]{EN00}. 
Quasi-compact representation of abelian semigroups have already been studied in \cite[Subsection~2.4]{DoGl2021}.
We prove the following spectral characterization; 
see \cite[Theorem~W.10]{DNP1987} for a similar result for powers of a single operator.

\begin{theorem}\label{charquasi}
	Let $E$ be a Banach space. 
	For a bounded representation $T \colon S \rightarrow \mathscr{L}(E)$ the following assertions are equivalent. 
	\begin{enumerate}
		\item\label{charquasia} 
		$T$ is quasi-compact.
		
		\item\label{charquasib} 
		$\sigma_{\unitary}(T)$ consists of Riesz points.
	\end{enumerate}
\end{theorem}

\cref{charquasi} and
\cref{chartotallyuniform}\ref{chartotallyuniform-erg} and~\ref{chartotallyuniform-poles} 
readily give the following consequence.

\begin{corollary}\label{qctu}
	Let $E$ be a Banach space.
	Every bounded and quasi-compact representation $T \colon S \rightarrow \mathscr{L}(E)$ 
	is totally uniformly mean ergodic with $\dim \ker(\chi-T) < \infty$ for all $\chi \in \sigma_{\unitary}(S)$.
\end{corollary}

For the proof of \cref{charquasi} we again need some auxiliary results.

\begin{lemma}\label{unimodeig}
	Let $E$ be a Banach space.
	For every quasi-compact bounded representation $T \colon S \rightarrow \mathscr{L}(E)$ 
	and $\chi \in \unitaryDual{S}$ we have $\dim \ker(\chi - T) < \infty$.
\end{lemma}

\begin{proof}
	There is some $s \in S$ such that $T_s \in \mathscr{L}(E)$ is a quasi-compact operator. 
	Hence, $\dim \ker(\lambda - T) < \infty$ for every $\lambda \in \T$ 
	(see \cite[Lemma~VIII.8.2]{DuSc1966} or \cite[Lemma~W.7]{DNP1987}). 
	Since $\ker(\chi - T) \subseteq \ker(\chi(s) - T_s)$ for every $\chi \in \unitaryDual{S}$, we obtain the claim.
\end{proof}

\begin{lemma}\label{compult}
	Let $E$ a Banach space, $J \neq \emptyset$ a set and $p$ an ultrafilter on $J$. 
	If $K\in \mathscr{L}(E)$ is compact, then $K^{p} \in \mathscr{L}(E^p)$ is compact, too.
\end{lemma}

\begin{proof}
	We write $\iota \colon E \rightarrow E^p$ for the canonical embedding, 
	$B$ for the open unit ball in $E$, and $B_p$ for the open unit ball in $E^p$. 
	We claim that the compact operator $K$ satisfies
	\begin{align*}
		K^p B_p \subseteq \iota \bigl(\overline{KB}\bigr).
	\end{align*}
	This shows that $K^p B_p$ is relatively compact in $E^p$, and hence $K^p$ is compact.
	
	So take a family $(x_j)_{j \in J} \in \ell^{\infty}(J,E)$ with $(x_j)_{j \in J}^p \in B_p$. 
	By definition of the quotient norm on $E^p$ we may assume that $\|(x_j)_{j \in J}\| < 1$. 
	By compactness of $K$, the limit $y \coloneqq \lim_{j \to p} Kx_j$ exists in $\overline{KB}$. 
	Therefore, $K^{p}(x_j)_{j \in J}^p = y^p = \iota(y)$. 
	This shows the claimed inclusion.
\end{proof}

\begin{proof}[Proof of \cref{charquasi}.]
	\ImpliesProof{charquasib}{charquasia}
	Assume first that (b) holds. 
	Take the decomposition $E = E_{\mathrm{r}} \oplus E_{\mathrm{s}}$ 
	from \cref{chartotallyuniform}\ref{chartotallyuniform-decomp}. 
	Then $E_{\mathrm{r}}$ is finite dimensional. 
	Let $P \in \mathscr{L}(E)$ denote the projection onto $E_{\mathrm{r}}$ along $E_{\mathrm{s}}$. 
	Then the net $(T_s (1-P))_{s \in S}$ converges to $0$ in operator norm. 
	Since $P$ has finite rank the operator $T_sP$ compact for each $s$, 
	which proves~\ref{charquasia}.
	
	\ImpliesProof{charquasia}{charquasib}
	Assume that $T$ is quasi-compact 
	and consider a non-empty set $J$ and an ultrafilter $p$ on $J$. 
	By \cref{compult} the representation $T^p \colon S \rightarrow \mathscr{L}(E^p)$ is also quasi-compact 
	and hence $\dim \ker(\chi - T^p) < \infty$ for every $\chi \in \unitaryDual{S}$ by \cref{unimodeig}. 
	But then \cref{thm:groh} implies that $\sigma_{\unitary}(T)$ consists of Riesz points.
\end{proof}

\subsection{The Niiro--Sawashima Theorem for representations}\label{subsec:nisa}

In this last subsection we focus on representations as positive operators on Banach lattices. 
For the theory of Banach lattices and positive operators we refer for instance 
to the classical monographs \cite{Me91, Schaefer1970} and, for a more introductory treatment, to \cite{Za97}. 
A result of Lotz and Schaefer (see \cite[Theorem~2]{LoSc1968} or \cite[Theorem V.5.5]{Schaefer1970}), 
based on ealier work by Niiro and Sawashima \cite[Main Theorem and Theorem~9.2]{NiSa1966}, 
asserts the following for a positive operator $T \in \mathscr{L}(E)$ on a Banach lattice $E$ 
with spectral radius $1$ and finite-dimensional spectral space for the spectral value $1$s: 
if $1$ is a pole of the resolvent of $T$, then all spectral values of modulus $1$ are also poles.
This was used by Lin to prove the following \cite[Theorem~1]{Lin78}: 
if a positive operator $T$ on a Banach lattice is power-bounded 
-- or, more generally, satisfies $\norm{T^n}/n \to 0$ as $n \to \infty$, -- 
then $T$ is quasi-compact if and only if the Cesàro means of its powers converge in operator norm to a finite-rank projection.

In \cref{nisa} we prove an analogue of those results 
for a \notion{positive} bounded representation $T \colon S \rightarrow \mathscr{L}(E)$ 
on a Banach lattice $E$, where \notion{positive} means that $T_s \in \mathscr{L}(E)$ is a positive operator for every $s\in S$. 
The following proposition is a key ingredient towards the theorem.

Since we deal exclusively with complex Banach spaces throughout the article, 
all Banach lattices in the sequel are complex, too.

\begin{proposition}\label{dominationfixspace}
	Let $E$ be a Banach lattice and let $T \colon S \rightarrow \mathscr{L}(E)$ be a representation 
	which is bounded, positive, and uniformly mean ergodic, and satisfies $\dim \fix(T) < \infty$. 
	Then 
	\begin{align*}
		\dim \ker(\chi - T) \leq \dim \fix(T)
	\end{align*}
	for every $\chi \in \unitaryDual{S}$.
\end{proposition}

The proof of the proposition uses the following two lemmas.
The first one, \cref{dom}, is a classical auxiliary result in Perron--Frobenius theory 
and its proof can be found in \cite[Lemma~C-III-3.11]{Nage1986}.

\begin{lemma}\label{dom}
	Let $E$ be a Banach lattice and $M$ and $L$ be two vector subspaces of $E$. 
	Assume that $f \in M$ implies $\vert f\vert \in L$. 
	Then $\dim M \leq \dim L$.
\end{lemma}

The second result, \cref{bina}, follows -- in case of a single operator -- 
from a decomposition theorem due to Bihlmaier (see \cite[Theorem~3.1]{Bihl2023}). 
We use similar methods to prove a general version for semigroups here.

\begin{lemma}\label{bina}
	Let $E$ be a Banach lattice 
	and let $T \colon S \rightarrow \mathscr{L}(E)$ be a positive and bounded representation. 
	Then there is an invariant closed linear subspace $F \subseteq E'$ 
	for the dual representation $T' \colon S \rightarrow \mathscr{L}(E')$ 
	and a norm $\|\cdot\|_F$ on $F$ equivalent to the norm induced by $E'$ such that
	\begin{enumerate}[(i)]
		\item \label{bina1}
		$F$ is a Banach lattice with the order inherited from $E$ and the norm $\|\cdot\|_F$,
		
		\item \label{bina2}
		$\ker(\chi - T') \subseteq F$ for every $\chi \in \unitaryDual{S}$, and
		
		\item \label{bina3}
		$T'|_F$ is a bounded representation of lattice isomorphisms.
	\end{enumerate}
\end{lemma}

\begin{proof}
	Let $\mathcal{S} \subseteq \mathscr{L}(E')$ be the closure of $T'(S)$ with respect to the weak* operator topology, 
	i.e., the locally convex topology generated by the seminorms
	\begin{align*}
		p_{x',x} \colon \mathscr{L}(E') \rightarrow \R_{\geq 0}, \quad \langle Rx', x\rangle
	\end{align*}
	for $x ' \in E'$ and $x \in E$. 
	Equipped with the weak* operator topology, $\mathcal{S}$ is compact (by the Banach--Alaoglu theorem) 
	and a semigroup with respect to composition of operators which is right topological, 
	i.e., the multiplication from the right
	\begin{align*}
		\mathcal{S} \rightarrow \mathcal{S}, \quad Q \mapsto QR
	\end{align*}
	is continuous for each $R \in \mathcal{S}$. 
	Moreover, all elements of $\mathcal{S}$ commute with each $T_s'$ for $s \in S$. 
	All those properties are straightforward to check;
	we refer for instance to \cite{Witz1964}, \cite{Koeh1994}, \cite{Koeh1995}, \cite{Roma2011}, 
	\cite{Roma2016}, \cite{Krei2018}, and \cite{Bihl2023} 
	for more details as well as for applications of this construction. 
	Note further that in our situation $\mathcal{S}$ consists of positive operators on $E'$.
	
	By the general structure theory of compact right topological semigroups (see, e.g., \cite[Theorem 3.11]{BeJuMi1989}), 
	$\mathcal{S}$ has a minimal idempotent $P \in \mathcal{S}$, 
	i.e., $P^2 = P$ and $P\mathcal{S}P$ is a group with neutral element $P$. 
	In particular, $P$ is a projection commuting with every $T_s'$ for $s \in S$, 
	and thus its range $F \coloneqq \mathrm{ran}(P)$ is invariant with respect to $T'$. 
	Let us show that $F$ satisfies the properties~\ref{bina1}--\ref{bina3}.
	
	\ref{bina1}
	Since $P$ is a positive projection, its range $F$ is a vector lattice 
	with respect to the order inherited from $E$ (but not necessarily with the same lattice operations) 
	and there exists an equivalent norm on $F$ that turns $F$ into a Banach lattice; 
	see \cite[Proposition III.11.5]{Schaefer1970}.
	
	\ref{bina2}
	Let $\chi \in \unitaryDual{S}$ and $x' \in \ker(\chi-T')$. 
	Since $P$ is in the closure of $T'(S)$ with respect to the weak* operator topology, 
	$Px' = \lambda x'$ for some $\lambda \in \T$. 
	But -- as $P$ is a projection -- we have $\lambda = 1$, hence $x' \in \mathrm{ran}(P) = F$. 
	
	\ref{bina3}
	Let $s \in S$. 
	The operator $T_s'P = PT_s'P$ has an inverse $R$ in the group $P\mathcal{S}P$. 
	The operator $R$ maps $F = \mathrm{ran}(P)$ into itself and satisfies $R T_s'P =T_s'RP = P$. 
	Thus, the restriction $R|_F$ is an inverse for the restriction $T_s|_{F}$. 
	Since $R$ is positive, $T_s|_{F}$ is a lattice isomorphism. 
	This implies \ref{bina3}.
\end{proof}

\begin{proof}[Proof of \cref{dominationfixspace}.]
	We first assume that $T_s \in \mathscr{L}(E)$ is a lattice homomorphism for every $s \in S$. 
	So take $\chi \in \unitaryDual{S}$ and $x \in \ker(\chi - T)$. 
	Then,
	\begin{align*}
		T_s|x| = |T_sx| = |\chi(s)x| = |x|
	\end{align*}
	for every $s \in S$, i.e., $|x| \in \fix(T)$. 
	Thus, \cref{dom} yields the claim.
	
	Now consider the general case. 
	Since $\mathscr{L}(E) \rightarrow \mathscr{L}(E'), \, R \mapsto R'$ is an isometric unital algebra anti-homomorphism 
	(i.e., it interchanges the order of multiplication)
	und maps $\overline{\mathrm{co}}\, T(S)$ to $\overline{\mathrm{co}}\, T'(S)$, 
	we obtain that the dual representation $T' \colon S \rightarrow \mathscr{L}(E')$, 
	and then also the bidual representation $T'' \colon S \rightarrow \mathscr{L}(E'')$, is uniformly mean ergodic. 
	In particular, by \cite[Theorem~8.36]{EFHN2015}, we obtain that $\fix(T)$ separates $\fix(T')$, 
	and that $\fix(T')$ separates $\fix(T'')$. 
	This implies $\dim \fix(T) = \dim \fix(T'')$.
	In view of the canonical embedding $j \colon E \rightarrow E''$ of $E$ into its bidual $E''$, 
	we only need to show that $\dim \ker(\chi - T'') \leq \dim \fix(T'')$. 
	By using \cref{bina} one sees that this estimate follows from the special case 
	for representations consisting of lattice homomorphisms that we considered at the beginning of the proof.
\end{proof}

\cref{dominationfixspace} allows us to deduce the following version of the Niiro-Sawashima theorem.

\begin{theorem}\label{nisa}
	Let $E$ be a Banach lattice. 
	For a positive bounded representation $T \colon S \rightarrow \mathscr{L}(E)$ 
	the following assertions are equivalent.
	\begin{enumerate}[(a)]
		\item\label{nisa1a} 
		$T$ is quasi-compact.
		
		\item\label{nisa1b} 
		$T$ is uniformly mean ergodic with $\dim \fix(T) < \infty$.
		
		\item\label{nisa1c} 
		The semigroup character $\mathds{1}_S \in \unitaryDual{S}$ is a Riesz point of $T$.
	\end{enumerate}
\end{theorem}

\begin{proof}
	\ImpliesProof{nisa1a}{nisa1b} 
	This implication follows from \cref{qctu}. 
	
	\ImpliesProof{nisa1b}{nisa1c}
	This implication follows from \cref{charunierg}\ref{charunierga} and~\ref{charuniergd}. 
	
	\ImpliesProof{nisa1c}{nisa1a}
	Assume that $\mathds{1}_S$ is a Riesz point of $T$. 
	Let $J$ be a non-empty set and $p$ and ultrafilter on $J$.  
	Then the fixed space $\mathrm{fix}(T^p)$ is finite dimensional by \cref{thm:groh}.  
	Moreover, since $\mathscr{L}(E) \to \mathscr{L}(E^p),~T\mapsto T^p$ is an isometric unitial algebra homomorphism 
	mapping $\overline{\mathrm{co}}\, T(S)$ to $\overline{\mathrm{co}}\, T^p(S)$, 
	the representation $T^p$ is uniformly mean ergodic. 
	Therefore, \cref{dominationfixspace} implies that $\dim \ker(\chi - T^p) < \infty$ for every $\chi \in \unitaryDual{S}$. 
	Using again \cref{thm:groh}, we conclude that every $\chi \in \sigma_{\unitary}(T)$ is a Riesz point of $T$. 
	By \cref{charquasi} this implies that the representation $T$ is quasi-compact. 
\end{proof}

\begin{remark}\label{rem:rotation-sg}
	Consider the special case $S = \big([0,\infty),+\big)$ and let $T \colon [0,\infty) \to \mathscr{L}(E)$ 
	be a bounded and strongly continuous representation on a Banach space $E$. 
	In other words, $T$ is a bounded $C_0$-semigroup. 
	In $C_0$-semigroup theory, \notion{uniform mean ergodicity} of $T$ usually means that the Cesàro means 
	\begin{align*}
		C_t \coloneqq \frac{1}{t} \int_0^t T_s \, \mathrm{d}s,
	\end{align*}
	where the integral is defined in the strong sense, converge with respect to the operator norm as $t \to \infty$. 
	This property is satisfied if and only if the range of the generator $A$ of $T$ is closed 
	if and only if the number $0$ is a pole of order $\le 1$
	of the resolvent of $A$ \cite[Theorem~V.4.10]{EN00}. 
	
	Despite the apparent similarity of these equivalences with the equivalences in Theorem~\ref{charunierg}, 
	one has to beware that the $C_0$-semigroup notion of uniform mean ergodicity is not equivalent 
	to the notion of uniform mean ergodicity from Definition~\ref{def:unif-mean-ergodic}. 
	The reason is that, in Definition~\ref{def:unif-mean-ergodic}, the mean ergodic projection $P$ 
	is assumed to be from the operator norm closure of the convex hull of $T(S)$. 
	However, while the Cesàro means $C_t$ converge in operator norm to a projection, 
	the $C_t$ themselves are only from the closure of $T(S)$ in the strong operator topology, in general, 
	since they are given by integrals that are defined strongly (i.e., pointwise). 
	
	As a concrete example, let $E = L^p(\T)$ for any $p \in [1,\infty)$ and let $T$ be the rotation semigroup in $E$, 
	i.e., $T_s f = f(\mathrm{e}^{\mathrm{i}s}\,\cdot\,)$ for each $s \in [0,\infty)$ and all $f \in L^p(\T)$. 
	Then $0$ is a first-order pole of the resolvent of the generator, so $T$ is uniformly mean ergodic in the sense of $C_0$-semigroup theory. 
	Moreover, the fixed space of $T$ consists of the constant functions only and is thus one-dimensional. 
	However, $T$ is clearly not quasi-compact. 
	Hence, the Niiro--Sawashima type \cref{nisa} shows that $T$ is not uniformly mean ergodic in the sense 
	of Definition~\ref{def:unif-mean-ergodic}.
	This also demonstrates that poles of the resolvent of the generator do not, in general, 
	translate to poles in the sense of Definition~\ref{def:pole-riesz}.
	
	It is worthwhile pointing out, though, that there is also a classical Niiro--Sawashima type theorem for generators 
	of positive $C_0$-semigroups \cite[Theorem~C-III-3.14]{Nage1986}: 
	if $T$ is a positive $C_0$-semigroup on a Banach lattice with generator $A$ and with, say, spectral bound $0$, 
	and $0$ is a pole of the resolvent, then all other spectral values of $A$ on the imaginary axis $\mathrm{i}\R$ 
	are poles of the resolvent, too. 
\end{remark}

\parindent 0pt
\parskip 0.5\baselineskip
\setlength{\footskip}{4ex}
\bibliographystyle{alpha}
\footnotesize

\bibliography{./bibliography} 

\begin{thebibliography}{EFHN15}

\bibitem[Arv74]{Arveson74}
W.~Arveson.
\newblock On groups of automorphisms of operator algebras.
\newblock {\em J. Funct. Anal.}, 15:217--243, 1974.

\bibitem[Bih23]{Bihl2023}
N.~Bihlmaier.
\newblock Compact right topological semigroups applied to operators.
\newblock 2023.
\newblock arXiv:2311.11920v1.

\bibitem[BJM89]{BeJuMi1989}
J.~F. Berglund, H.~Junghenn, and P.~Milnes.
\newblock {\em Analysis on Semigroups. Function Spaces, Compactifications,
  Representations}.
\newblock Wiley, 1989.

\bibitem[BP92]{BaPh1992}
C.~J.~K. Batty and V.~Q. Ph\'{o}ng.
\newblock Stability of strongly continuous representations of abelian
  semigroups.
\newblock {\em Math. Z.}, 209(1):75--88, 1992.

\bibitem[BP04]{BaPr04}
B.~Basit and A.~J. Pryde.
\newblock Ergodicity and stability of orbits of unbounded semigroup
  representations.
\newblock {\em J. Aust. Math. Soc.}, 77:209--232, 2004.

\bibitem[DG21]{DoGl2021}
A.~Dobrick and J.~Gl\"{u}ck.
\newblock Uniform convergence of operator semigroups without time regularity.
\newblock {\em J. Evol. Equ.}, 21(4):5101--5134, 2021.

\bibitem[DJT95]{DiJaTo1995}
J.~Diestel, H.~Jarchow, and A.~Tonge.
\newblock {\em Absolutely Summing Operators}.
\newblock Cambridge University Press, 1995.

\bibitem[DL75]{DoLi75}
Y.~Domar and L.-A. Lindahl.
\newblock Three spectral notions for representations of commutative {Banach}
  algebras.
\newblock {\em Ann. Inst. Fourier}, 25(2):1--32, 1975.

\bibitem[DNP87]{DNP1987}
R.~Derndinger, R.~Nagel, and G.~Palm.
\newblock Ergodic {T}heory in the {P}erspective of {F}unctional {A}nalysis.
\newblock Unpublished book manuscript, 1987.

\bibitem[DS66]{DuSc1966}
N.~Dunford and J.~T. Schwartz.
\newblock {\em Linear Operators. Part I: General Theory}.
\newblock Interscience Publishers, 1966.

\bibitem[EFHN15]{EFHN2015}
T.~Eisner, B.~Farkas, M.~Haase, and R.~Nagel.
\newblock {\em Operator Theoretic Aspects of Ergodic Theory}.
\newblock Springer, 2015.

\bibitem[Eme07]{Em07}
E.~Y. Emel'yanov.
\newblock {\em Non-spectral Asymptotic Analysis of One-Parameter Operator
  Semigroups}.
\newblock Birkh{\"a}user, 2007.

\bibitem[EN00]{EN00}
K.~Engel and R.~Nagel.
\newblock {\em One-Parameter Semigroups for Linear Evolution Equations}.
\newblock Springer, 2000.

\bibitem[GG83]{GrGr83}
G.~Greiner and U.~Groh.
\newblock A {Perron} {Frobenius} theory for representations of locally compact
  {Abelian} groups.
\newblock {\em Math. Ann.}, 262:517--528, 1983.

\bibitem[GG19]{GerGlu2019}
M.~Gerlach and J.~Glück.
\newblock Convergence of positive operator semigroups.
\newblock {\em Trans. Amer. Math. Soc.}, 372:6603--6627, 2019.

\bibitem[GH19]{GlHa2019}
J.~Gl\"uck and M.~Haase.
\newblock Asymptotics of operator semigroups via the semigroup at infinity.
\newblock In {\em Positivity and Noncommutative Analysis}, pages 167--203.
  Birkh\"auser/Springer, 2019.

\bibitem[Gro83]{Groh1983}
U.~Groh.
\newblock On the peripheral spectrum of uniformly ergodic positive operators on
  {C*}-algebras.
\newblock {\em J. Operator Theory}, 10:31--37, 1983.

\bibitem[Hei80]{Hein1980}
S.~Heinrich.
\newblock Ultraproducts in {B}anach space theory.
\newblock {\em J. Reine Angew. Math.}, 313:72--104, 1980.

\bibitem[Hua95]{Huang95}
S.~Huang.
\newblock Stability properties characterizing the spectra of operators on
  {Banach} spaces.
\newblock {\em J. Funct. Anal.}, 132(2):361--382, 1995.

\bibitem[Kre85]{Kren1985}
U.~Krengel.
\newblock {\em Ergodic Theorems}.
\newblock De Gruyter, 1985.

\bibitem[Kre18]{Krei2018}
H.~Kreidler.
\newblock Compact operator semigroups applied to dynamical systems.
\newblock {\em Semigroup Forum}, 97:523--547, 2018.

\bibitem[Kö94]{Koeh1994}
A.~Köhler.
\newblock {\em Enveloping Semigroups in Operator Theory and Topological
  Dynamics}.
\newblock PhD thesis, University of Tübingen, 1994.

\bibitem[Kö95]{Koeh1995}
A.~Köhler.
\newblock Enveloping semigroups for flows.
\newblock {\em Proc. Roy. Irish Acad. Sect.}, 95:179--–191, 1995.

\bibitem[Lin78]{Lin78}
M.~Lin.
\newblock Quasi-compactness and uniform ergodicity of positive operators.
\newblock {\em Isr. J. Math.}, 29:309--311, 1978.

\bibitem[LMF73]{LyMaFe73}
Y.~I. Lyubich, V.~I. Machaev, and G.~M. Fel'dman.
\newblock On representations with a separable spectrum.
\newblock {\em Funct. Anal. Appl.}, 7:129--136, 1973.

\bibitem[LS68]{LoSc1968}
H.~P. Lotz and H.~H. Schaefer.
\newblock \"{U}ber einen {S}atz von {F}. {N}iiro und {I}. {S}awashima.
\newblock {\em Math. Z.}, 108:33--36, 1968.

\bibitem[Lyu71]{Lyubich71}
Y.~I. Lyubich.
\newblock On the spectrum of a representation of an abelian topological group.
\newblock {\em Sov. Math., Dokl.}, 12:1482--1486, 1971.

\bibitem[MN91]{Me91}
P.~Meyer-Nieberg.
\newblock {\em Banach Lattices}.
\newblock Springer, 1991.

\bibitem[Mur90]{Murr1990}
G.~J. Murphy.
\newblock {\em C*-Algebras and Operator Theory}.
\newblock Academic Press, 1990.

\bibitem[Nag86]{Nage1986}
R.~Nagel, editor.
\newblock {\em One-parameter Semigroups of Positive Operators}.
\newblock Springer, 1986.

\bibitem[Nag01]{Nagy01}
A.~Nagy.
\newblock {\em Special Classes of Semigroups}.
\newblock Kluwer Academic Publishers, 2001.

\bibitem[NS66]{NiSa1966}
F.~Niiro and I.~Sawashima.
\newblock On the spectral properties of positive irreducible operators in an
  arbitrary {B}anach lattice and problems of {H}. {H}. {S}chaefer.
\newblock {\em Sci. Papers College Gen. Ed. Univ. Tokyo}, 16:145--183, 1966.

\bibitem[Rom11]{Roma2011}
A.~Romanov.
\newblock Weak* convergence of operator means.
\newblock {\em Izv. Math.}, 75:1165--–1183, 2011.

\bibitem[Rom16]{Roma2016}
A.~Romanov.
\newblock Ergodic properties of discrete dynamical systems and enveloping
  semigroups.
\newblock {\em Ergodic Theory Dynam. Systems}, 36:198–--214, 2016.

\bibitem[Sch70]{Schaefer1970}
H.~H. Schaefer.
\newblock {\em Banach Lattices and Positive Operators}.
\newblock Springer, 1970.

\bibitem[Wit64]{Witz1964}
K.~G. Witz.
\newblock Applications of a compactification for bounded operator semigroups.
\newblock {\em Illinois J. Math.}, 8:685--696, 1964.

\bibitem[Wol84]{Wolff1984}
M.~P.~H. Wolff.
\newblock Spectral theory of group representations and their nonstandard hull.
\newblock {\em Isr. J. Math.}, 48:205--224, 1984.

\bibitem[Zaa97]{Za97}
A.~C. Zaanen.
\newblock {\em Introduction to Operator Theory in {Riesz} Spaces}.
\newblock Springer, 1997.

\end{thebibliography}
\footnotesize

\end{document}